\documentclass[12pt,reqno,a4paper]{amsart}            

\usepackage{amsmath,amsfonts,amssymb}
\usepackage{mathtools}
\usepackage{color}
\usepackage{cite}
\usepackage{pifont}
\usepackage{graphicx}
\usepackage{mathptmx}
\usepackage{hyperref}

\usepackage[marginparwidth=40pt]{geometry}
\allowdisplaybreaks

\newcounter{lookcounter} 

\newcommand{\Hausdorff}{\mathrm{H}}  
\newcommand{\Dir}{\mathrm{D}}  
\newcommand{\Neu}{\mathrm{N}}  

\newtheorem{theorem}{Theorem}[section]
\newtheorem{lemma}[theorem]{Lemma}
\newtheorem{corollary}[theorem]{Corollary}
\theoremstyle{definition}
\newtheorem{remark}[theorem]{Remark}
\newtheorem{definition}[theorem]{Definition}

\newcommand{\HS}{\mathcal{H}}           
\newcommand{\A}{\mathcal{A}}    
\renewcommand{\a}{\mathfrak{a}}    

\newcommand{\clo}[1]{\overline{{#1}}} 
\newcommand{\abs}[2][{}]{\lvert{#2}\rvert_{{#1}}}    

\newcommand{\R}{\mathbb{R}} 
\newcommand{\N}{\mathbb{N}} 
\newcommand{\Z}{\mathbb{Z}} 

\renewcommand{\L}{\mathsf{L}_2} 
\renewcommand{\H}{\mathsf{H}^1} 
\newcommand{\Ho}{\mathsf{H}^1_0} 
\newcommand{\CC}{\mathsf{C}^{\infty}} 

\newcommand{\Id}{\mathrm{I}} 

\newcommand{\dd}{\, \mathrm d} 
 
\newcommand{\e}{_\eps}
\newcommand{\eps}{\varepsilon}
\newcommand{\ie}{_{i\eps}}
\newcommand{\I}{\mathcal{I}\e}
\renewcommand{\phi}{\varphi}

\newcommand{\ds}{\displaystyle}
\newcommand{\cupl}{\bigcup\limits}
\newcommand{\suml}{\sum\limits}
\newcommand{\intl}{\int\limits}
    
\newcommand{\maxl}{\max\limits}

\DeclareMathOperator{\supp}   {supp}
\DeclareMathOperator{\dom}    {dom}
\DeclareMathOperator{\dist}   {dist}
\DeclareMathOperator{\diam}   {diam}
\DeclareMathOperator{\capty}  {cap}

\newcommand{\pp}{{1\prime}}

\newcommand{\restr}[1]{{\restriction}_{#1}} 

\def\XXint#1#2#3{{\setbox0=\hbox{$#1{#2#3}{\int}$ }
\vcenter{\hbox{$#2#3$ }}\kern-.6\wd0}}

\begin{document}
\parskip=3mm
\parindent=0mm

\title[Operator estimates for the crushed ice problem]
{Operator estimates for the crushed ice problem}

\author{Andrii Khrabustovskyi}%
\address{Institut f\"ur Numerische Mathematik, Technische Universit\"at Graz, 8010 Graz, Austria} %
\email{khrabustovskyi@math.tugraz.at}

\author{Olaf Post}
\address{Fachbereich 4 -- Mathematik,
  Universit\"at Trier,
  54286 Trier, Germany}
\email{olaf.post@uni-trier.de}

\subjclass[2010]{Primary 58J50; Secondary 35B27, 35P15, 47A10}

\begin{abstract} 
  Let $\Delta_{\Omega\e}$ be the Dirichlet Laplacian in the domain
  $\Omega\e:=\Omega\setminus\left(\cup_i D_{i\eps}\right)$. Here
  $\Omega\subset\mathbb{R}^n$ and $\{D_{i\eps}\}_{i}$ is a family of
  tiny identical holes (``ice pieces'') distributed periodically in $\mathbb{R}^n$ with
  period $\eps$.  We denote by $\capty(D_{i\eps})$ the capacity of a
  single hole.  It was known for a long time that $-\Delta_{\Omega\e}$
  converges to the operator $-\Delta_{\Omega}+q$ in strong resolvent
  sense provided the limit $q:=\lim_{\eps\to 0}
  \capty(D_{i\eps})\eps^{-n}$ exists and is finite.  In the current
  contribution we improve this result deriving estimates for the rate
  of convergence in terms of operator norms. As an application, we
  establish the uniform convergence of the corresponding semi-groups
  and (for bounded $\Omega$) an estimate for the difference of the
  $k$-th eigenvalue of $-\Delta_{\Omega\e}$ and
  $-\Delta_{\Omega\e}+q$.  Our proofs relies on an abstract scheme for
  studying the convergence of operators in varying Hilbert spaces
  developed previously by the second author.  \bigskip

  \noindent{\emph{Keywords:} crushed ice problem; homogenization; norm
    resolvent convergence; operator estimates; varying Hilbert spaces}
\end{abstract}

\maketitle

\section{Introduction}
\label{sec:intro}

In the current work we revisit one of the classical problems in homogenization theory -- homogenization of the Dirichlet Laplacian in a domain with a lot of tiny holes. It is also known as \emph{crushed ice problem}.
Below, we briefly recall the setting of this problem and the main result. 

Let $\Omega$ be an open domain in $\R^n$ ($n\geq 2$) and
$\left\{D\ie\right\}_i$ be a family of small holes. The holes are
identical (up to a rigid motion) and are distributed evenly in
$\Omega$ along the $\eps$-periodic cubic lattice -- see
Figure~\ref{fig1}. We set
\begin{gather*}
  \Omega\e := \Omega \setminus\Bigl( \bigcup_{i}\clo{ D\ie}\Bigr).
\end{gather*}
The domain $\Omega\e$ is depicted in Figure~\ref{fig1}. More precise
description of this domain will be given in the next section.

\begin{figure}[h]
\begin{picture}(160,155)
\includegraphics{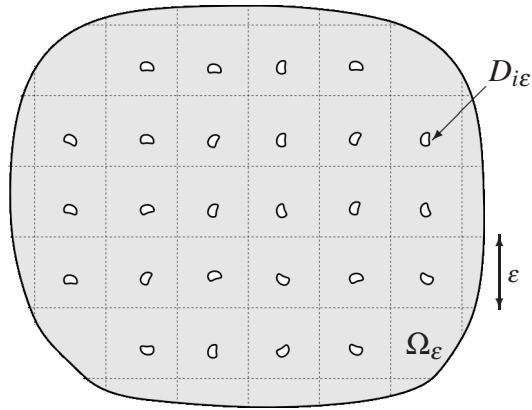}
\put(5,48){\vector(0,1){17}}
\put(5,48){\vector(0,-1){11}}
\put(8,43){$^\eps$}
\put(1,123){$D\ie$}
\put(0,121){\vector(-1,-1){20}}
\put(-30,20){$\Omega\e$}
\end{picture}
\caption{The domain $\Omega\e$ obtained from $\Omega$ by removing the
  obstacles $D\ie$.  To avoid technical problems with the boundary of
  $\Omega$, the obstacles are only placed into cells which lie
  entirely in $\Omega$.}
\label{fig1}
\end{figure}

In $\Omega\e$ we study the following problem:
\begin{gather*}
-\Delta_{\Omega\e} u\e + u\e = f\restr{\Omega\e},
\end{gather*}
where $\Delta_{\Omega\e}$ is the Dirichlet Laplacian in $\Omega\e$, $f\in \L(\Omega)$ is a given function, $f\restr{\Omega\e}$ is the restriction of $f$ to $\Omega\e$.
The goal is to describe the behaviour of the solution $u\e$ to this problem as $\eps\to 0$.

It turns out that the result depends on the limit $q:=\lim_{\eps\to
  0} \capty(D\ie)\eps^{-n}$ being finite or infinite (here
$\capty(D\ie)$ is the capacity of a single hole,
see~\eqref{capacity} for details). Namely, if $q=\infty$ then
$\|u\e\|_{\L(\Omega\e)}\to 0$ as $\eps\to 0$. Otherwise, if
$q<\infty$, $\|u\e-u\|_{\L(\Omega\e)}\to 0$ as $\eps\to 0$, where
$u$ is the solution to the problem
\begin{equation*}
  -\Delta_\Omega u +qu+ u = f.
\end{equation*}
This result was proven independently by V.A.~Marchenko,
E.Ya.~Khruslov~\cite{MK64} (the case $q<\infty$), J.~Rauch,
M.~Taylor~\cite{RT75} (the cases $q=0$ and $q=\infty$) and
D.~Cioranescu, F.~Murat~\cite{CM82} (all scenario) by using different
tools  --- potential theory, capacitary methods and variational approach (the so-called Tartar's energy method), respectively.  J.~Rauch and M.~Taylor also treated the case
of randomly distributed holes under assumptions resembling the case
$q>0$ in a deterministic case; the pioneer result in this direction
was obtained by M.~Kac in~\cite{Kac74}, who investigated the case of
{uniformly} distributed holes.

Note, that this result remains valid if on the external boundary
(i.e.\ on $\partial\Omega \setminus\left( \bigcup_{i} { \partial
    D\ie}\right)$) one imposes Neumann, Robin, mixed or any other
\emph{$\eps$-independent} boundary conditions (then $-\Delta_\Omega$
is the Laplace operator subject to these conditions on
$\partial\Omega$).

Besides the resolvent convergence one can study the convergence of
spectrum or the convergence of the semi-group
$\exp(\Delta_{\Omega\e}t)$.  In the later case the name \emph{crushed
  ice problem} is indeed reasonable\footnote{\label{foot1}Let us
  assume that $\Omega$ is an isolated container occupied by a
  homogeneous medium, while the sets $D\ie$ are regarded as a small
  pieces of ice. Under a certain idealization (the ice pieces do not
  melt and move) the heat distribution in $\Omega\e$ at time $t>0$ is
  described by the function $\exp(\Delta_{\Omega\e} t)u_0$, where
  $\Delta_{\Omega\e}$ is the Laplace operator subject to Dirichlet
  conditions on the boundary of the ice pieces and Neumann conditions
  on $\partial\Omega$ (since the container is isolated), $u_0$ is the
  heat distribution at $t=0$.}. Also domains with a lot of Dirichlet
holes have interesting scattering properties (fading/solidifying
obstacles, cf.~\cite{R75b,RT75}).

For more details on the topic we refer also to
articles~\cite{Ba88,Be95,Kh72,Oza83,R75a,PV80}, as well as to the
monographs~\cite{Br02,Cha84,CPS07,MK74,MK06,Si05}.

In what follows, we focus on the case $q<\infty$.

In the language of operator theory one can reformulate the above
result as follows: the operator $-\Delta_{\Omega\e}$ converges to the
operator $-\Delta_\Omega+q$ in strong resolvent sense. Strictly
speaking, we are not able to treat the classical resolvent convergence
(since the underlying operators act in different Hilbert spaces), but
we have its natural analogue for varying domains with $\Omega_\eps
\subset \Omega$:
\begin{gather}\label{strong}
\forall f\in \L(\Omega):\ \|(-\Delta_{\Omega\e}+\Id)^{-1}J\e f-J\e(-\Delta_{\Omega}+q+\Id)^{-1}f\|_{\L(\Omega\e)}\to 0\text{ as }\eps\to 0,
\end{gather}
where $J\e f:= f\restr{\Omega\e}$.

In the recent preprint~\cite{CDR17} the authors
improved~\eqref{strong} by proving (a kind of) \emph{norm} resolvent
convergence, namely
\begin{gather}\label{uniform}
\|J'\e(-\Delta_{\Omega\e}+\Id)^{-1}-(-\Delta_{\Omega}+q+\Id)^{-1}J'\e\|_{\mathcal{L}(\L(\Omega\e),\L(\Omega))}\to 0\text{ as }\eps\to 0,
\end{gather}
where $J'\e\colon\L(\Omega\e)\to \L(\Omega)$ is the operator of
extension by zero.  The authors assumed that $D\ie$ are balls,
distributed $\eps$-periodically in $\Omega$. 
For bounded $\Omega$ their proof resembles the variational approach developed in~\cite{CM82}, for unbounded $\Omega$ they also utilize a rapid decay of the Green's function of $-\Delta+\Id$. 

In the current work we extend the result of~\cite{CDR17} providing an
estimate for the rate of convergence in~\eqref{uniform} (see
Theorem~\ref{thm:main:2} below). We also improve~\eqref{strong} (see
Theorem~\ref{thm:main:1}) deriving the operator estimate
\begin{gather*}
\|(-\Delta_{\Omega\e}+\Id)^{-1}J\e-J\e(-\Delta_{\Omega}+q+\Id)^{-1}\|_{\mathcal{L}(\L(\Omega),\L(\Omega\e))}\leq 4\delta\e,
\end{gather*} 
where
$\delta\e=\left|q-\lim_{\eps\to 0}\capty(D\e)\eps^{-n}\right|+\gamma\e$
with $\gamma\e=o(1)$ depending on the dimension $n$ (for the
``physical'' cases $n=2$ and $n=3$ one has
$\gamma\e=\mathcal{O}(\eps\ln\eps)$ and $\gamma\e=\mathcal{O}(\eps)$,
respectively).

As a consequence of our main results, we establish uniform convergence
of the corresponding semi-groups and (for bounded $\Omega$) an
estimate for the difference between the $k$-th eigenvalue of
$-\Delta_{\Omega\e}$ and $-\Delta_{\Omega}+q$ --- see
Theorems~\ref{thm:main:3}--\ref{thm:main:4}.

Let us stress that in all our results (except Theorem~\ref{thm:main:4}) we do not assume that the domain $\Omega$ is \emph{bounded}.

Our proofs are based on the abstract scheme for studying the convergence of
operators in varying Hilbert spaces which was developed by the second author of the present article in~\cite{P06} and in more detail in the monograph~\cite{P12}.

Before proceeding to the main part of the work let us mention several
related results:
\begin{itemize}

\item Some estimates for the rate of convergence in~\eqref{strong}
  were obtained in~\cite[{\S}16]{CPS07}. Namely, assuming that $n=3$, $D\ie$ are balls of radius $\eps^3$ (that is
  $\capty(D\ie)\eps^{-3}=4\pi=q$) distributed $\eps$-periodically, and
  the function $f$ belongs to the H\"older class
  $\mathsf{C}^{0,\a}(\overline{\Omega})$, the authors derived the
  estimates
\begin{gather}
  \nonumber
  \|(-\Delta_{\Omega\e}+\Id)^{-1}J\e f-J\e(-\Delta_{\Omega}+q+\Id)^{-1}f\|_{ \L(\Omega\e)}
  \leq {C}\eps\|f\|_{\mathsf{C}^{0,\a}(\overline{\Omega})}.
  \\\label{H1est}
  \|(-\Delta_{\Omega\e}+\Id)^{-1}J\e f-J\e \varphi\e (-\Delta_{\Omega}+q+\Id)^{-1}f\|_{ \H(\Omega\e)}\leq {C}\eps\|f\|_{\mathsf{C}^{0,\a}(\overline{\Omega})},
\end{gather}
where $\varphi\e$ is the operator of multiplication by a certain
cut-off function.

\item \eqref{H1est}-like estimates were also obtained
  in~\cite{Be95}. In this work the holes are distributed
  $\eps$-periodically in a bounded domain $\Omega\subset\mathbb{R}^n$
  ($n\geq 2$), no special assumptions on the geometry of holes are
  imposed.  Let $f \in \mathsf L_\infty(\Omega)$. Then one has the
  estimate
 \begin{gather*}
 \|(-\Delta_{\Omega\e})^{-1}J\e f-J\e \psi\e (-\Delta_{\Omega}+q)^{-1}f\|_{ \H(\Omega\e)}\leq
 \Lambda\e\|f\|_{ \mathsf{L}_\infty(\Omega)},\ \lim_{\eps\to 0}\Lambda\e= 0,
 \end{gather*}
 where the small factor $\Lambda\e$ is expressed in terms of the first eigenvalue of the Laplace
 operator on a period cell subject to the Dirichlet conditions on the hole boundary and the periodic
 conditions on the external part of the period cell boundary; the  function $\psi\e$ is built on the
 basis of the corresponding eigenfunction.
 
\item One can also study a surface distribution of holes, i.e.\ holes
  being located near some hypersurface $\Gamma$ intersecting $\Omega$.
  This problem was first considered in~\cite{MK64}; it was proved that
  the limit operator is $-\Delta_\Omega+q\delta_\Gamma$.  Here
  $q\in\mathsf{L}_\infty(\Gamma)$ is a positive function, and
  $\delta_\Gamma$ is a delta-distribution supported on $\Gamma$. For
  the case $n=2$, the norm resolvent convergence with estimates on the
  rate of convergence were obtained in~\cite{BCD16}, where even more
  general elliptic operators were treated.  The proofs in~\cite{BCD16}
  rely on variational formulations for the pre-limit and the
  homogenized resolvent equations (the key object of their analysis is
  a certain integral identity associated with the difference of the
  resolvents).  Note that the method we use in the current works
  allows to treat surface distributions of holes as
  well.  Nevertheless, to simplify the presentation, we focus on the
  bulk distribution of holes only.
  
\item Operator estimates in homogenization theory is a rather young
  topic.  The classical homogenization problem concerning elliptic
  operators of the form
  \begin{gather*} 
    \A\e = -\mathrm{div}\left(A\Bigl(\frac . \eps\Bigr)\nabla\right),
    \text{where $A(\cdot)$ is a $\mathbb{Z}^n$-periodic function,}
  \end{gather*}
  was first treated in~\cite{BS04,BS07,G04,G06,Z05a,Z05b,ZP05}, see
  the recent papers~\cite{ZP16,S16} for further references.  In
  particular, the article~\cite{Z05b} deals with a perturbation which
  is defined by rescaling an abstract periodic measure.  The technique
  developed in~\cite{Z05b} can be applied for deriving operator
  estimates is the case of periodically perforated domains provided
  the sizes of holes and distances between them are of the same
  smallness order (evidently, this does not hold for the problem we
  study in the current paper).  Operator estimates were also obtained
  for elliptic operators with frequently alternating boundary
  conditions (see, e.g.,~\cite{BBC10}), for problems in domains with
  oscillating boundary~\cite{BCFP13}, or for the ``double-porosity''
  model in~\cite{CK17}. For more results we refer to the
  paper~\cite{BCD16} containing a comprehensive overview on operator
  estimates in homogenization theory.

\item In~\cite{anne-post:pre17} we treat (possibly non-compact)
  manifolds with an increasing (even infinite) number of balls removed
  (similarly as in~\cite{RT75}), and show operator estimates using
  similar methods as in this article.

\end{itemize}

\subsection*{Acknowledgements}

This research was carried on when the first author was a postdoctoral researcher in Karlsruhe Institute of Technology. He gratefully acknowledges financial support by the Deutsche
Forschungsgemeinschaft (DFG) through CRC~1173 ``Wave phenomena: analysis and numerics''.

\section{Setting of the problem and main results}
\label{sec:main}

Let $n\geq 2$ and let $\Omega \subset \R^n$ be a domain (not
necessarily bounded) with $\mathsf{C}^2$-boundary $\partial\Omega$.
We also assume that there exists a constant $\theta_\Omega>0$
such that the following map is injective on $\partial\Omega \times [0, \theta]$ provided
$\theta<\theta_\Omega$:
\begin{gather}\label{tau}
(x, t) \mapsto x +  t \nu(x)
\end{gather}
where $\nu \colon \partial\Omega\to \mathbb{S}^{n-1}$ the unit
inward-pointing normal vector field on $\partial\Omega$.

Additionaly, we require $\Omega$ to be \emph{uniformly regular} in the sense of Browder \cite{Br59}. This requirement is automatically fulfilled, for example, for domains with compact smooth boundaries or for compact, smooth perturbations of half-spaces. Under this assumption the Dirichlet Laplacian in $\Omega$ defined via
$$\dom(\Delta_\Omega)=\mathsf{H}^2(\Omega)\cap \H_0(\Omega),\quad \Delta_\Omega f:=\Delta f$$
is a self-adjoint operator (see, e.g., the recent paper \cite{BLLR17} for more details and references on this issue).

We note, that our results remain valid under less restrictive
assumptions on $\partial\Omega$, see Remark~\ref{remark-smooth} below.

In what follows we denote by $C$, $C_1$ etc.\ generic constants
depending only on the dimension $n$.

We set $\square:=(-1/2,1/2)^n$.

\begin{figure}[h]
  \setlength{\unitlength}{1mm}
  \begin{picture}(60,60)
    \includegraphics{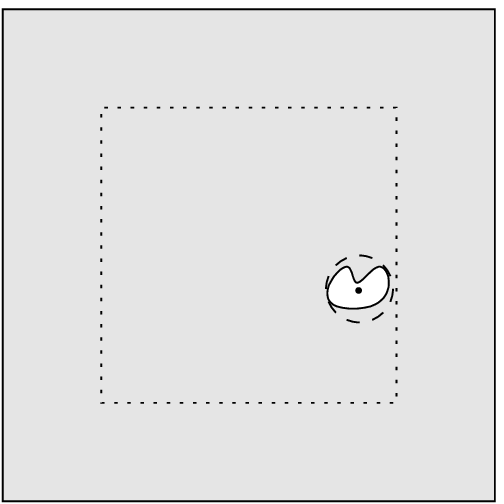}
    \put(-55,25){\vector(0,1){25}}
    \put(-55,25){\vector(0,-1){25}}
    \put(-58,25){$\eps$}
    \put(-15,35){\vector(1,0){5}}
    \put( 5, 35){\vector(-1,0){5}}
    \put(-8,34){$\kappa \eps$}
    \put(1,0){$\square\ie$}
    \put(-20,28){$D\ie$}
    \put(-18,27){\vector(1,-2){2.5}}
    \put(-30,14){$B(D\ie)$}
    \put(-22,17){\vector(2,1){4.8}}
  \end{picture}
  \qquad
  \begin{picture}(60,60)
    \includegraphics{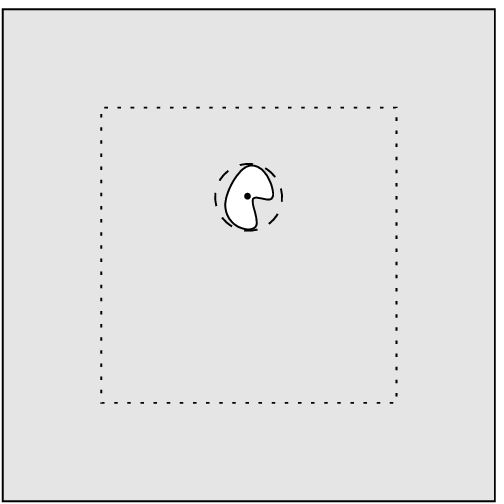}
    \put(-15,35){\vector(1,0){5}}
    \put( 5, 35){\vector(-1,0){5}}
    \put(-8,34){$\kappa \eps$}
    \put(1,0){$\square_{j\eps}$}
    \put(-35,24){$D_{j\eps}$}
    \put(-30.5,26){\vector(1,1){4.5}}
    \put(-25,19){$B(D_{j\eps})$}
    \put(-22,23){\vector(-1,3){1.8}}
  \end{picture}
  \caption{Two scaled cells $\square\ie$ and $\square_{j\eps}$ and
    possible positions of the obstacles $D\ie$ and $D_{j\eps }$
    (white).  The smallest ball $B(D\ie)$ (dashed circle) containing
    the obstacle $D\ie$ has security distance $\kappa\eps$ from the
    boundary of $\square\ie$, i.e., it should stay inside the dotted
    cube of side length $(1-2\kappa)\eps$.}
  \label{fig2}
\end{figure}
Now we describe a family of holes in $\Omega$ (see
Figure~\ref{fig2}).  Let $D\e$ be a Lipschitz domain in $\R^n$
depending on a small parameter $\eps>0$. We denote by $d\e$ the radius
of the smallest ball containing $D\e$. It is assumed that
\begin{gather}\label{d}
  (d\e)^{n-2}
  \leq C\eps^n
  \quad\text{as $n\geq 3$,}\quad
  |\ln d\e|^{-1}
  \leq C\eps^2
  \quad\text{as $n=2$,}\quad
  \text{or}\\
  \nonumber
  d\e \le
  \begin{cases}
    C_1 \eps^{n/(n-2)},& \text{as $n \ge 3$},\\
    \exp(-1/(C\eps^2)),& \text{as $n=2$}
  \end{cases}
\end{gather}
(hence, in particular, $d\e=o(\eps)$). 
For $i \in\mathbb{Z}^n$, let $D\ie$ be a set enjoying the following properties:
\begin{gather} 
  \nonumber
D\ie\text{ coinsides with }D\e\text{ up to a rigid motion},\\ 
  \nonumber
B(D\ie)\subset \square\ie:=\eps(\square+i),\\ 
\label{eq:dist.obst}
\mathrm{dist}\big(B(D\ie),\partial\square\ie\big)\geq\kappa\eps\text{ for some }\kappa>0,   
\end{gather}
where $B(D\ie)$ is the smallest ball containing $D\ie$ (the radius of this ball is $d\e$). 

Finally, we set 
\begin{gather*}
  \Omega\e := \Omega \setminus\left( \bigcup_{i \in \I}\clo{ D\ie}\right),
\end{gather*}
where 
\begin{gather*}
\I :=\left\{i \in \Z^n:\ \square\ie \subset \Omega\right\},
\end{gather*}
i.e.\ the set of those indices for which the rescaled unit cell
$\square\ie$ is entirely in $\Omega$ (with positive distance to
$\partial\Omega$).  The domain $\Omega\e$ is depicted in
Figure~\ref{fig1}.

By $\A\e$ we denote the Dirichlet Laplacian on $\Omega\e$, i.e.\
the operator acting in the Hilbert space $\L(\Omega_\eps)$ associated with the closed densely defined positive sesquilinear form 
\begin{gather*}
  \a\e[u,v]
  := \int_{\Omega\e} \nabla u\cdot\nabla \bar v \dd x,
  \qquad
  \dom(\a\e):= \Ho(\Omega_\eps).
\end{gather*}

Our goal is to describe  the behaviour of the resolvent $(\A\e+\Id)^{-1}$ as $\eps\to 0$ under the assumption that the following limit exists and is finite:
\begin{gather}\label{q}
q=\lim_{\eps\to 0}\frac {\capty(D\e)}{\eps^{n}},
\end{gather}
where $\capty(D\e)$ is the capacity of the set $D\e$. Recall (see, e.g.,~\cite{T96}), that
for $n\geq 3$ the capacity of a set $D\subset\R^n$ is defined via
\begin{gather}
  \label{capacity}
  \capty(D)=\int_{\R^n}|\nabla H(x)|^2\dd x,
\end{gather}
where $H$ is a solution to the problem
\begin{gather}
  \label{cap-problem}
  \begin{cases}
    \Delta H(x)=0,&x\in\R^n\setminus\clo{D},\\
    H(x)=1,&x\in\partial D,\\
    H(x)\to 0,&|x|\to\infty.
  \end{cases}
\end{gather}
One has also the following variational characterization of the capacity, namely
\begin{gather}
  \label{cap-min}
  \capty(D)=\min\int_{\R^n}|\nabla u(x)|^2\dd x,
\end{gather}
where the minimum is taken over $u\in \mathsf{C}_0^\infty(\R^n)$ being
equal to $1$ on a neighbourhood of $D$.

For $n=2$ the right-hand-side of~\eqref{cap-min} is zero for an arbitrary domain $D$, hence we need a modified definition. It is as follows:
\begin{gather}
  \label{capacity2}
  \capty(D)=\int_{B_1}|\nabla H(x)|^2\dd x,
\end{gather}
where $B_1$ is the unit ball concentric with $B(D)$ -- the smallest
ball containing $D$ (here we assume that the set $D$ is small enough
so that $D\subset B(D)\subset B_1$), $H$ solves the problem
\begin{gather}\label{cap-problem2}
  \begin{cases}
    \Delta H(x)=0,&x\in B_1\setminus\overline{D},\\
    H(x)=1,&x\in \partial D,\\
    H(x)=0,&x\in\partial B_1.
\end{cases}
\end{gather}

Further, proving the main results, we will use the following pointwice
estimates for the functions $H$ at some positive distance from
$B(D)$, see~\cite[Lemma~2.4]{MK06}.
\begin{lemma}
\label{lemma-Hest}
Let $x\in\R^n\setminus\overline{B(D)}$. We denote by $\rho(x)$ the
distance from $x$ to $B(D)$, and by $d$ the radius of
$B(D)$. One has:  
\begin{gather*}
\begin{array}{lll}
\abs{H(x)}\leq 
\dfrac{C \cdot d^{n-2}}{(\rho(x))^{n-2}},&
\abs{\nabla H(x)}\leq 
\dfrac{C \cdot d^{n-2}}{(\rho(x))^{n-1}}&\text{as }n\geq 3,\\[2ex]
\abs{H(x)}\leq 
\dfrac{C |\ln d|^{-1}}{|\ln \rho(x)|^{-1} },&
\abs{\nabla H(x)}\leq 
\dfrac{C |\ln d|^{-1}}{\rho(x)}&\text{as }n=2
\end{array}
\end{gather*}
provided $\rho(x)\geq C_0 d$ as $n\geq 3$ or
$\rho(x)\geq \exp\left(-C_0\sqrt{|\ln d|}\right)$ as $n\geq 3$,\, for some $C_0>0$.  
\end{lemma}

\begin{remark}\label{remark-cap}
Due to~\eqref{q} one has
\begin{gather}\label{O}
\capty(D\e)\eps^{-n}=\mathcal{O}(1).
\end{gather}
In fact, this condition also follows directly from~\eqref{d}. 
Indeed, using the monotonicity of the capacity,  we get $\capty(D\e)\leq \capty(B\e)$, where $B\e$ is ball of radius $d\e$ containing $D\e$.
For this ball the function $H$ can be computed explicitly:
\begin{equation*}
  H(x) = \dfrac{(d\e)^{n-2}}{|x|^{n-2}}\text{ as }n\geq 3,\quad
  H(x) = \dfrac{\ln x}{\ln d\e}\text{ as }n=2,
\end{equation*} 
hence $\capty(B\e)=(n-2)|\mathbb{S}^{n-1}|(d\e)^{n-2}$ as $n\geq 3$ and $\capty(B\e)=2\pi |\ln d\e|^{-1}$ as $n=2$, hence, due to~\eqref{d}, we get~\eqref{O}.
\end{remark}

Finally, we introduce the limiting operator $\A$. It acts in $\L(\Omega)$ and is defined by
$$\A : =-\Delta_\Omega+q\Id,\quad\dom(\A):=\dom(\Delta_\Omega)=\mathsf{H}^2(\Omega)\cap\Ho(\Omega).$$
By $\a$ we denote the associated form:
\begin{gather*}
  \a[u,v]
  := \int_{\Omega} \left(\nabla u\cdot\nabla \bar v + q\, u\bar v\right)\dd x,
  \qquad
  \dom(\a):= \Ho(\Omega).
\end{gather*}

The operators $\A\e$ and $\A$ act in different Hilbert spaces, namely
$\HS\e:=\L(\Omega\e)$ and $\HS:=\L(\Omega)$, respectively. Therefore
we are not able to apply the usual notion of resolvent convergence
and thus a suitable modification is needed. There are many ways how to
do this in a ``smart'' way. For example (cf.~\cite{IOS89,V81}), one
can treat the behaviour of the operator
\begin{equation*}(\mathcal{A\e}+\Id)^{-1}J\e-J\e(\A+\Id)^{-1}\in\mathcal{L}(\HS, \HS\e),
\end{equation*}
where $J\e \colon \HS\to \HS\e$ is a suitable bounded linear operator satisfying
\begin{gather}\label{as-un}
\forall f\in\HS:\ \|J\e f\|_{\HS\e}\to \|f\|_{\HS}\text{ as }\eps\to 0.
\end{gather} 

It is natural to
choose the operator $J\e$ as the operator of restriction to $\Omega\e$, i.e.\  
\begin{gather}\label{J}
(J\e f)(x)= 
\begin{cases}
f(x),&x\in\Omega\e,\\
0,&x\in\ds\cupl_{i\in \I }D\ie.
\end{cases}
\end{gather}
Due to~\eqref{d} one has for each compact set $K\subset\R^n$
\begin{gather}\label{totalD} 
\ds\suml_{i:\,D\ie\subset K}|D\ie| \leq \alpha\e|K|,\ \alpha\e=o(1)\text{ does not depend on }K,
\end{gather}
where $|K|$ stands for the Lebesgue measure of $K$.   Hence, evidently, \eqref{as-un} holds. The results of
\cite{CM82,MK64,RT75} can be reformulated as follows:
\begin{equation*}\forall f\in\HS\colon\quad\lim_{\eps\to 0}\left\|(\mathcal{A\e}+\Id)^{-1}J\e f-J\e(\A+\Id)^{-1} f\right\|_{\HS\e}\to 0, \end{equation*}
i.e.\ one has a kind of strong resolvent convergence.

Now, we can state our main result.

\begin{theorem}
\label{thm:main:1}
One has  
\begin{gather*}
\left\|(\A\e+\Id)^{-1}J\e -J\e(\A+\Id)^{-1} \right\|_{\mathcal{L}(\HS,\HS\e)}\leq 4\delta\e,
\end{gather*}
where $\delta\e$ is defined by 
\begin{gather}\label{delta}
\delta\e=|\capty(D\e)\eps^{-n}-q|+C_{\Omega,\kappa,\beta}\cdot
\begin{cases}
\eps|\ln\eps|,& n=2,\\[1ex]
\eps,& n=3\\[1ex]
\ds\eps^{1-\beta},\ \beta>0, & n=4,\\[1ex]
\maxl\left\{\eps;\, {d\e}{\eps}^{-1}\right\}, & n\geq 5,
\end{cases}
\end{gather}  
and the constant $C_{\Omega,\kappa,\beta}$ depends on the domain $\Omega$, the relative distance $\kappa$ of the obstacles from the period cell boundary (see~\eqref{eq:dist.obst}),  and, in the case $n=4$, on $\beta$.
\end{theorem}

\begin{remark}
Via the same arguments as in Remark~\ref{remark-cap} one gets
$ (d\e)^{n-2}\eps^{-n}\geq C>0$ provided $q>0$, hence, using the definition of $\delta\e$,  we obtain   
\begin{equation*}
  q>0,\ n\geq 5:\quad
  \delta\e=|\capty(D\e)\eps^{-n}-q|+C_{\Omega,\kappa}\eps^{2/(n-2)}.
\end{equation*}
\end{remark}

Let $J\e' \colon \HS\e\to\HS$ be the operator of extension by zero:
\begin{gather}\label{J'}
(J\e'u)(x)=
\begin{cases}
  u(x),&x\in\Omega\e,\\
  0,&x\in \ds\cupl_{i\in\I}D\ie.
\end{cases}
\end{gather}
Then the main result of~\cite{CDR17} is equivalent to
\begin{equation*}\left\|J'\e(\mathcal{A\e}+\Id)^{-1}-(\A+\Id)^{-1}J'\e \right\|_{\mathcal{L}(\HS\e,\HS)}\to 0,\ \eps\to 0.\end{equation*}
The next theorem gives an improvement of this statement. 

\begin{theorem}\label{thm:main:2}
One has
\begin{gather*}
\left\|J'\e(\mathcal{A\e}+\Id)^{-1}-(\A+\Id)^{-1}J'\e \right\|_{\mathcal{L}(\HS\e,\HS)}\leq 6\delta\e,
\end{gather*}  
where $\delta\e$ is defined in~\eqref{delta}. Moreover, 
\begin{gather*}
\left\|J'\e(\mathcal{A\e}+\Id)^{-1}J\e-(\A+\Id)^{-1}\right\|_{\mathcal{L}(\HS,\HS)}\leq 9\delta\e,\\
\left\|(\mathcal{A\e}+\Id)^{-1}-J\e(\A+\Id)^{-1}J'\e \right\|_{\mathcal{L}(\HS\e,\HS\e)}\leq 13\delta\e.
\end{gather*} 
\end{theorem}

One important applications of the norm resolvent convergence is the uniform convergence of semi-groups generated by $\A\e$ and $\A$. Namely, we can approximate $\exp(-\A\e t)$ in terms of simpler operators $\exp(-\A t)$, $J\e$ and $J'\e$:

\begin{theorem}\label{thm:main:3}
One has for each $t>0$:
\begin{equation*}\left\| {\exp(-\A\e t)} - J\e{\exp(-\A t)}J\e'\right\|_{\mathcal{L}(\HS\e,\HS\e)}\leq c_t\delta\e,\end{equation*} where $\delta\e$ is defined in~\eqref{delta}, and the constant $c_t$ depends only on $t$.
\end{theorem}

Another important application is the Hausdorff convergence of spectra,
 see~\cite{CDR17}.  
Using Theorem~\ref{thm:main:1} we are able to extend this
result by obtaining an estimate for the difference between the corresponding  eigenvalues. Namely, let the domain $\Omega$ be bounded. We denote by $\{\lambda_{k,\eps}\}_{k\in\mathbb{N}}$ and
$\{\lambda_{k}\}_{k\in\mathbb{N}}$ the sequences of the eigenvalues of $\A\e$ and $\A$, respectively,
arranged in the ascending order and repeated according to their
multiplicities.
 
\begin{theorem}\label{thm:main:4}
For each $k\in\mathbb{N}$ one has 
\begin{equation}\label{spectrum1}
\lim_{\eps\to 0}\lambda_{k,\eps}=\lambda_k, 
\end{equation}
moreover
\begin{equation}
\label{spectrum2}
|\lambda_{k,\eps}-\lambda_k|\leq
  4C\e(\lambda_{k,\eps}+1)(\lambda_k+1)\delta\e,
\end{equation} 
where $\delta\e$ is defined
in~\eqref{delta}, and  $|C\e|\leq C$, $\lim_{\eps\to 0}C\e=1$.
\end{theorem} 

In the next section we introduce an abstract scheme, which then will be applied for the proof of the above theorems.

\section{Abstract framework}
\label{sec:abstract}

In this section we present an abstract scheme for studying the convergence of
operators in varying Hilbert spaces. It was developed by the second author of the present article in~\cite{P06} and in more detail in the
monograph~\cite{P12} (see also the later work~\cite{MNP13}, where non-self-adjoint operators were treated).

Let $\HS$ and $\HS\e$ be two separable Hilbert spaces. 
Note, that within this section $\HS\e$ is just a \textit{notation} for some Hilbert space, which (in general) differs from the space $\HS$, i.e.\ the sub-index $\eps$ does not mean that this space depends on a small parameter. Of course, further we will use the results of this section for $\eps$-dependent space $\HS\e=\L(\Omega\e)$.

Let $\a$ and $\a\e$ be closed, densely defined, non-negative
sesquilinear forms in $\HS$ and $\HS\e$, respectively. We denote by
$\A$ and $\A\e$ the non-negative, self-adjoint operators associated
with $\a$ and $\a\e$, respectively.
  
Associated with the operator $\A$, we can introduce a natural
\emph{scale of Hilbert spaces} $\HS^k$ defined via the \emph{abstract
  Sobolev norm}:
\begin{gather*} 
\HS^k=\dom \A^{k/2},\quad  
  \|f\|_{\HS^k}
  \coloneqq \|f\|_k 
  \coloneqq \|{(\A+\Id)^{k/2}f}\|_{\HS}.
\end{gather*}
In particular, we have $\HS^0=\HS$ with $\|f\|_{\HS^0}= \|f\|_{\HS}$,
$\HS^1=\dom(\a)$ with
$\|f\|_{\HS^1}=(\a[f,f]+\|f\|_{\HS}^2)^{1/2}$, and
$\HS^2=\dom(\A)$ with $\|f\|_{\HS^2}=\|\A f+f\|_{\HS}$.

Similarly, we denote by
$\HS\e^{k}$ the scale of Hilbert spaces associated with $\A\e$.
The corresponding norms will be denoted by $\|\cdot\|_{\HS\e^k}$.

We now need pairs of so-called \emph{identification} or
\emph{transplantation operators} acting on the Hilbert spaces and
later also pairs of identification operators acting on the form
domains.
 
\begin{definition}[({see~\cite[App.]{P06} or~\cite[Ch.~4]{P12}})]
  \label{def:deltaclo}
  Let $\delta\e \ge 0$ and $k \in \N$.  Moreover, let $J\e\colon
  \HS\to {\HS\e}$ and ${J'\e}\colon {\HS\e}\to \HS$ be linear bounded
  operators.  In addition, let ${J\e^1} \colon {\HS^1} \to {\HS\e^1}$
  and $ {J\e^{1\prime}} \colon {\HS\e^1}\to {\HS^1}$ be linear bounded
  operators on the form domains.  We say that $(\HS,\a)$ and
  $(\HS\e,\a\e)$ are \emph{$\delta\e$-close of order $k$ with respect
    to the operators $J\e$, $J'\e$, $J^1\e$, $J^{1\prime}\e$}, if the
  following conditions hold:
\begin{gather}
\label{cond1a}\tag{C$_{1\text a}$}
\|J\e f-J\e^1 f\|_{\HS\e}\leq \delta\e\|f\|_{\HS^1},\ \forall f\in \HS^1,
\\[1ex]\label{cond1b}\tag{C$_{1\text b}$}
\|J'\e u-J\e^{1\prime} u\|_{\HS}\leq 
\delta\e\|u\|_{\HS\e^1},\ \forall u\in \HS\e^1,
\\\label{cond2}\tag{C$_2$}
\left|(J\e f,u)_{\HS\e}-(f, J'\e u)_{\HS}\right|\leq 
\delta\e \|f\|_{\HS}\|u\|_{\HS\e},\ \forall f\in \HS,u\in \HS\e,
\\\label{cond4a}\tag{C$_{3\text a}$}
\|J\e f\|_{\HS\e}\leq (1+\delta\e)\|f\|_{\HS},\ \forall f\in \HS,
\\\label{cond4b}\tag{C$_{3\text b}$}
\|J'\e u\|_{\HS}\leq (1+\delta\e)\|u\|_{\HS\e},\ \forall u \in \HS\e,
\\\label{cond3a}\tag{C$_{4\text a}$}
\|f-J'\e J\e f\|_{\HS}\leq \delta\e\|f\|_{\HS^1},\ \forall f\in \HS^1,
\\\label{cond3b}\tag{C$_{4\text b}$}
\|u-J\e J'\e u\|_{\HS\e}\leq \delta\e\|u\|_{\HS\e^1},\ \forall u\in \HS\e^1,
\\\label{cond5}\tag{C$_5$}
\left|\a\e(J^1\e f,u)-\a(f, J\e^{1\prime} u)\right|\leq 
\delta\e \|f\|_{\HS^k}\|u\|_{\HS\e^1}, \forall f\in \HS^k, u\in \HS\e^1.
\end{gather}
\end{definition}

\begin{remark}
  For $\delta\e=0$ the definition above implies that the operators
  $\A$ and $\A\e$ are unitary
  equivalent. Indeed,~\eqref{cond2}--\eqref{cond3b} assure that the
  operator $J\e$ is unitary with the inverse $J'\e$; due
  to~\eqref{cond1a}--\eqref{cond1b} $J^1\e$ and $J^{1\prime}\e$ are the
  restrictions of $J\e$ and $J^1\e$ onto $\dom(\a)$ and $\dom(\a\e)$,
  respectively. Hence, in view of~\eqref{cond5}, $J\e$ realises the
  unitary equivalence of $\A$ and $\A\e$.
\end{remark}

Now, we present the main implications of the definition of $\delta\e$-closeness.
 
\begin{theorem}[{\cite[Th.~A.5]{P06}}]
  \label{thm:abstract:1}
One has 
\begin{gather*} 
\left\|(\A\e+\Id)^{-1}J\e -J\e(\A+\Id)^{-1} \right\|_{\mathcal{L}(\HS,
\HS\e)}\leq 4\delta\e,
\end{gather*}
provided conditions \eqref{cond1a}, \eqref{cond1b}, \eqref{cond2}, and \eqref{cond5} hold with $k\leq 2$.
\end{theorem}

\begin{remark}
  Let $\A\e$ ($\eps>0$), $\A$ be non-negative self-adjoint operators
  in the \emph{same} Hilbert space $\HS$, and let $\a\e$ and $\a$ be
  the corresponding sesquilinear forms. We assume that
  $\dom(\a\e)=\dom(\a)$ and
  \begin{gather}\label{remark-kato1}
    \left|\a\e( f,u)-\a(f,  u)\right|\leq 
    \delta\e \sqrt{\a[f,f]+\|f\|_{\HS}}\sqrt{\a\e[u,u]+\|u\|_{\HS}},\ \forall f,u\in \dom(\a),
  \end{gather}
  where $\delta\e\to 0$ as $\eps\to 0$.  Due to~\eqref{remark-kato1}
  $(\HS,\a)$ and $(\HS,\a\e)$ are $\delta\e$-close of order $1$ with
  respect to the \emph{identity} maps $J\e,J\e'$ (on $\HS$) and
  $J\e^1,J\e^{1\prime}$ (on $\dom(\a)$).  Then by
  Theorem~\ref{thm:abstract:1}
  \begin{gather}\label{remark-kato2}
    \left\|(\A\e+\Id)^{-1} -(\A+\Id)^{-1} \right\|_{\mathcal{L}({\HS})}\to 0 .
  \end{gather} 
  In fact, it would suffice for~\eqref{remark-kato2}
  if~\eqref{remark-kato1} is satisfied whenever $f=u$, see
  Theorem~VI.3.6 in T.~Kato's monograph~\cite{Kat66}.  In this sense,
  Theorem~\ref{thm:abstract:1} can be regarded as a generalization of
  this classical result to the setting of varying spaces.
\end{remark}

\begin{theorem}[{\cite[Th.~A.8]{P06}}]
  \label{thm:abstract:2}
Let $U\subset\clo{\R_+}$ be an open set containing either $\sigma(\A)$ or $\sigma(\A\e)$. Let $\psi \colon \clo{\R_+}\to\mathbb{C}$ be a bounded measurable function, continuous on $U$ and such that the limit $\lim_{\lambda\to\infty}\psi(\lambda)$ exists. 

Then there exists $ \eta_\psi(\delta)>0$ with $\eta_\psi(\delta)\to 0$ as $\delta\to 0$ such that
\begin{gather}\label{eta}
\left\|\psi(\A\e)J\e -J\e\psi(\A) \right\|_{\mathcal{L}(\HS,
\HS\e)}\leq \eta_\psi(\delta\e)
\end{gather}
for all pairs $(\HS,\a)$ and $(\HS\e,\a\e)$, which are $\delta\e$-close of order $k\leq 2$.

\end{theorem}

\begin{remark}\label{rem-examples}
The important example of the function $\psi$ satisfying the requirements of the above theorem is $\psi(\lambda)=\exp(-\lambda t)$, $t>0$ is a parameter. 
Another important example is the function $\psi=\mathbf{1}_{(\alpha,\beta)}$ -- the characteristic function of the interval $(\alpha,\beta)$ with $\alpha,\beta\notin\sigma(\A)$ or $\alpha,\beta\notin\sigma(\A\e)$.  In this case 
Theorem~\ref{thm:abstract:2} gives the closeness of the spectral projections.
\end{remark} 

\begin{theorem}[{\cite[Th.~A.10]{P06}}]
  \label{thm:abstract:3}
  Let for some function $\psi$ the estimate \eqref{eta} be valid.  Then
\begin{gather*} 
\left\|J'\e\psi(\A\e) -\psi(\A)J'\e\right\|_{\mathcal{L}(\HS\e,
\HS)}\leq \eta_\psi(\delta\e)+2\|\psi\|_{\infty}\delta\e,
\\ 
\left\|J\e'\psi(\A\e)J\e -\psi(\A) \right\|_{\mathcal{L}(\HS,
\HS)}\leq C_\psi\delta\e + 2\eta_\psi(\delta\e),\\ 
\left\|\psi(\A\e) -J\e\psi(\A)J'\e \right\|_{\mathcal{L}(\HS\e,
\HS\e)}\leq 5 C_\psi\delta\e + 2\eta_\psi(\delta\e)
\end{gather*}
provided~\eqref{cond2}--\eqref{cond3b} hold true. Here $\eta_\psi(\delta\e)$ comes from \eqref{eta}, $\|\cdot\|_\infty$ stands for the $\mathsf{L}_\infty$-norm, and $C_\psi$ is a constant satisfying $|\psi(\lambda)|\leq C_\psi(1+\lambda)^{-\frac12}$ for all $\lambda\geq 0$.
\end{theorem}

For $\psi(\lambda)=(1+\lambda)^{-1}$ 
one has $\eta_{\psi}(\delta\e)=4\delta\e$ (see Theorem~\ref{thm:abstract:1}), $C_\psi=1$, and hence we immediately get the following corollary from Theorem~\ref{thm:abstract:3}.

\begin{corollary}
\label{abs-cor}
One has
\begin{gather*}
\left\|J'\e(\mathcal{A\e}+\Id)^{-1}-(\A+\Id)^{-1}J'\e \right\|_{\mathcal{L}(\HS\e,\HS)}\leq 6\delta\e,
\\ 
\left\|J'\e(\mathcal{A\e}+\Id)^{-1}J\e-(\A+\Id)^{-1}\right\|_{\mathcal{L}(\HS,\HS)}\leq 9\delta\e,\\ 
\left\|(\mathcal{A\e}+\Id)^{-1}-J\e(\A+\Id)^{-1}J'\e \right\|_{\mathcal{L}(\HS\e,\HS\e)}\leq 13\delta\e
\end{gather*} 
provided $(\HS,\a)$ and $(\HS\e,\a\e)$ are $\delta\e$-close of order $k\leq 2$. 
\end{corollary} 
 
For ``good enough'' functions  the last statement of Theorem~\ref{thm:abstract:3} can be improved. Evidently the function
$\psi(\lambda)=\exp(-\lambda t)$ ($t>0$) satisfies the requirements of
the theorem below.

\begin{theorem}[{\cite[Th.~3.7]{MNP13}}\footnote{Actually, \eqref{refined-estimate} is proven in~\cite[Th.~3.7]{MNP13}  only for the case $k=1$. For $k=2$ the proof is repeated word-by-word since it relies only on the last estimate in Corollary~\ref{abs-cor}.}\label{thm-holom}]
    Let $\Sigma_\Theta:=\{z\in\mathbb{C}:\ |\arg (z+1)|<\Theta\}$ with
    $\Theta\in (0,\pi)$ and $\psi:\Sigma_\Theta\to\mathbb{C}$ be a
    holomorphic function satisfying $\psi(z)=\mathcal{O}(|z|^{-\mu})$
    for some $\mu>\frac12$. Let $(\HS,\a)$ and $(\HS\e,\a\e)$ be
    $\delta\e$-close of order $k\leq 2$.  Then
\begin{gather}\label{refined-estimate} 
\left\|\psi(\A\e) -J\e\psi(\A)J'\e \right\|_{\mathcal{L}(\HS\e,
\HS\e)}\leq c_\psi\delta\e,
\end{gather}
where $c_\psi$ is a constant depending on $\psi$.
\end{theorem}

\begin{remark}
In fact, \eqref{refined-estimate} is valid even for less regular functions. For instance, it holds for $\psi=\mathbf 1_{(\alpha,\beta)}$ as in Remark~\ref{rem-examples}, see~\cite[Sec.~4.5, Cor.~4.5.15]{P12}.
\end{remark}

The last result concerns the convergence of spectra in general.  For
two compact sets $X,Y\subset\R$ we denote by $\dist_\Hausdorff(X,Y)$
the \emph{Hausdorff distance} between these sets, i.e.\
\begin{equation*}
  \dist_\Hausdorff(X,Y)
  =\max \Bigl\{
          \sup_{x\in X}\, \dist(x,Y);\,\sup_{y\in Y} \,\dist (y,X)
        \Bigr\},
\end{equation*}
where $\dist(x,Y)=\inf_{y \in Y} \abs{x-y}$.
 
\begin{remark}\label{remark-hausd}
Let $\{X\e\subset\R\}\e$ be a family of compact domains and 
\begin{gather}\label{hausdorff}
\dist_\Hausdorff\left(X\e,X\right)\to 0\text{ as }\eps\to 0
\end{gather}
for some compact domain $X\subset\R$. It is easy to prove (see, e.g.,~\cite[Proposition~A.1.6]{P12}) that~\eqref{hausdorff} holds iff 
the following two conditions are fulfilled:
\begin{itemize}
\item[(i)] Let $\lambda_0\in \mathbb{R}\setminus X$. Then there exists $d>0$ such that $X\e\cap\{\lambda:\ |\lambda-\lambda_0|<d\}=\varnothing$. 

\item[(ii)] Let $\lambda_0\in X$. Then there exists a family
  $\{\lambda\e\}\e$ with $\lambda_\eps \in X_\eps$ such that
  $\lim_{\eps\to 0}\lambda\e=\lambda_0$.

\end{itemize}
\end{remark}

\begin{theorem}[{\cite[Th.~A.13]{P06}}]
  \label{thm:abstract:4}
  There exists $\eta(\delta)>0$ with $\eta(\delta)\to 0$ as $\delta\to 0$ such that
\begin{gather*}
  \dist_\Hausdorff
  \left(\frac 1{1+\sigma(\A)},\frac 1 {1+\sigma(\A\e)}\right)
  \leq\eta(\delta\e) 
\end{gather*}
for all pairs   $(\HS,\a)$ and $(\HS\e,\a\e)$  which are $\delta\e$-close of some order $k\in\mathbb{N}$. 
\end{theorem}

\section{Proof of the main results}
\label{sec:proofs}

For an open subset $M\subset\R^n$ ($M \neq \emptyset$) we denote
by $\langle f\rangle_M$ 
the \emph{mean value} of $f$ over $M$, i.e.\
\begin{equation*}
  \langle f\rangle_M:= \ds\frac 1{|M|} \intl_M f(x) \dd x.
\end{equation*}

Recall that $\HS\e$ and $\HS$ stand for the spaces $\L(\Omega\e)$ and
$\L(\Omega)$, respectively; $\a\e$ and $\a$ are the sesqulilinear
forms associated with the operators $\A\e$ and $\A$.  Also, recall
that $\HS^1\e$ (respectively, $\HS^1$) is a Hilbert space of functions
from $\dom(\a\e)$ (respectively, $\dom(\a)$) equipped with the scalar
product $(u,v)_{\HS^1\e}=\a\e[u,v]+(u,v)_{\HS\e}$ (respectively,
$(u,v)_{\HS^1}=\a[u,v]+(u,v)_{\HS}$).

Our goal is to show that $(\HS,\a)$ and $(\HS\e,\a\e)$ are
$\delta\e$-close of order $k=2$ with respect to the operators $J\e
\colon \HS\to\HS\e$ defined in~\eqref{J}, $J'\e \colon \HS\e\to\HS$
defined in~\eqref{J'} and suitable operators $J^1\e \colon
\HS^1\to\HS^1\e$, $J'\e \colon \HS^1\e\to\HS^1$.  Then
Theorem~\ref{thm:main:1} follows immediately from
Theorem~\ref{thm:abstract:1}, Theorem~\ref{thm:main:2} follows from
Corollary~\ref{abs-cor}, and Theorem~\ref{thm:main:3} follows from
Theorem~\ref{thm-holom}.  The proof of Theorem~\ref{thm:main:4} needs
an additional step. For convenience, we postpone it to the end of this
section.

We define the operator $J\e^\pp$ being equal to $J\e'$ on $\HS^1_\eps$.
Thus the only non-obvious definition is the one of $J^1\e$
as we have to assure that $J^1\e f \restr{\cup_{i\in\I}D_{i\eps}}=0$.

We define
\begin{equation*}
  J^1\e f
 := f - \suml_{i\in\I}P\ie f - \suml_{i\in\I}Q\ie f
\end{equation*}
with
\begin{equation*}
  (P\ie f)(x)
  := \bigl(f(x)-f\ie\bigr) \chi\ie(x)
\text{\qquad and\qquad }  
  (Q\e f)(x)
  := f\ie H\ie(x) \widehat \chi\ie(x)
\end{equation*}
\begin{figure}[h]
  \setlength{\unitlength}{1mm}
  \begin{picture}(60,80)
    \includegraphics{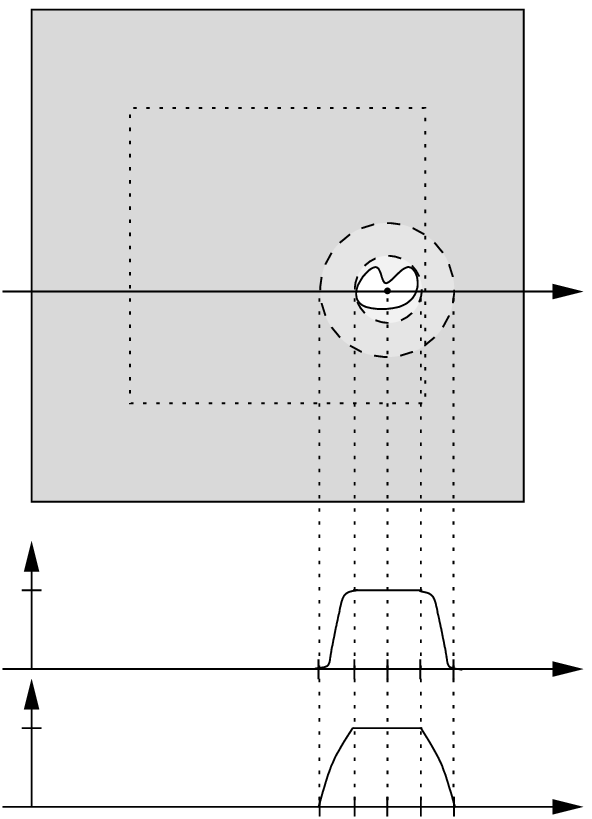}
    \put(-5,32){$\square\ie$}
    \put(-55,34){$Y\ie$}
    \put(-25,62){$D\ie$}
    \put(-22,61){\vector(1,-2){3.5}}
    \put(-33,58){$x\ie$}
    \put(-28.5,58){\vector(2,-1){8}}
    \put(-41,49){$B(D\ie)$}
    \put(-29,50){\vector(2,1){5}}
    \put(1,53){$x$}
    \put(1,14.7){$x$}
    \put(-55,26){$\chi\ie(x)$ ($n\ge3$)}
    \put(-60,22){$1$}
    \put(1,0.5){$x$}
    \put(-55,11.5){$\chi\ie(x)$ ($n=2$)}
    \put(-60,7.5){$1$}
    \put(-28,27){\vector(1,0){4}}
    \put(-16,27){\vector(-1,0){4}}
    \put(-23.5,28){$d_\eps$}
    \put(-23.5,18){\vector(1,0){3.5}}
    \put(-23.5,18){\vector(-1,0){3.5}}
    \put(-25,19.5){$2d_\eps$}
    \put(-23.5,3){\vector(1,0){3.5}}
    \put(-23.5,3){\vector(-1,0){3.5}}
    \put(-23.5,4){$\eps^2$}
  \end{picture}
  \qquad
  \begin{picture}(60,80)
    \includegraphics{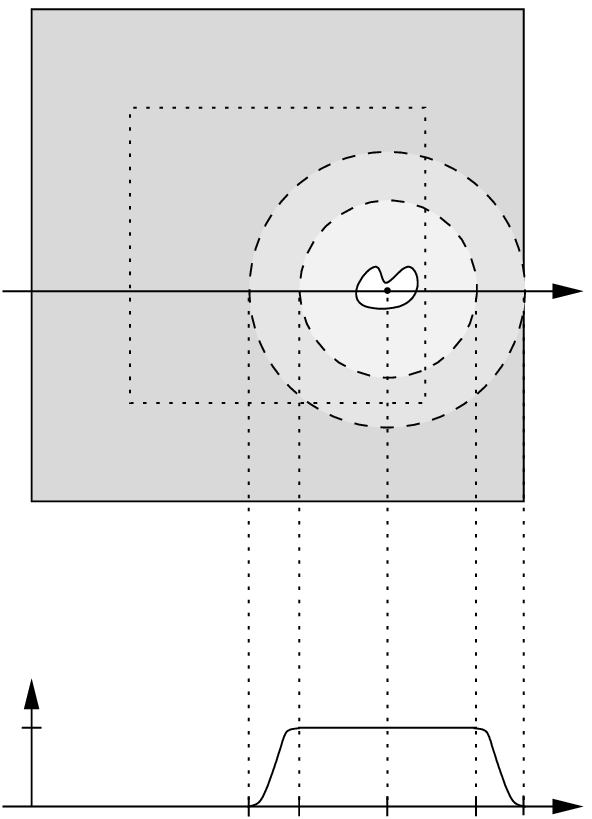}
    \put(-5,32){$\square\ie$}
    \put(-55,34){$Y\ie$}
    \put(-23,59){$D\ie$}
    \put(-19,58){\vector(0,-1){3.5}}
    \put(-40,58.7){$x\ie$}
    \put(-35,58.7){\vector(3,-1){14}}
    \put(1,53){$x$}
    \put(-55,11.5){$\widehat \chi\ie(x)$}
    \put(-60,7.5){$1$}
    \put(1,0.5){$x$}
    \put(-24.5,25){\vector(1,0){4.5}}
    \put(-24.5,25){\vector(-1,0){4.5}}
    \put(-29,27){$\frac{\kappa\eps}2+d_\eps$}
    \put(-24.5,15){\vector(1,0){4.5}}
    \put(-24.5,15){\vector(-1,0){9.5}}
    \put(-28.5,17){$\kappa \eps + d_\eps$}
  \end{picture}
  \caption{The two cut-off functions $\chi\ie$ and $\widehat \chi\ie$
    with decay on the scale $d_\eps$ and $\eps$, respectively.  On the
    left, there is the cut-off function $\chi\ie$, which is $1$ inside
    the small ball $B(D\ie)$ (light gray) with radius $d_\eps$, and
    $0$ outside the larger ball around $x\ie$ with radius $2d_\eps$
    ($n\ge 3$) resp.\ $\eps^2$ ($n=2$).  \newline On the right, there
    is the cut-off function $\widehat \chi\ie$, which is $1$ inside
    the light gray ball of radius $\kappa\eps/2+d_\eps$, and $0$ on
    the dark gray area outside the larger ball of radius $\kappa
    \eps+d_\eps$.  Both cut-off functions have support in
    $\square\ie$.}
  \label{fig3}
\end{figure}
Here (see also Figure~\ref{fig3})
\begin{itemize}
\item 
$f\ie := \langle f\rangle_{\square\ie}$,

\item
  $x\ie$ denotes the center of the smallest ball $B(D\ie)$
  containing the set $D\ie$ (recall that this ball has radius $d\e$),
\item for $n\geq 3$:
  \begin{equation*}
    \chi\ie(x)=\chi\left(\dfrac{|x-x\ie|}{d\e}\right),
  \end{equation*}
  where $\chi\in\CC(\R)$ is a smooth cut-off function such that
  $|\chi(t)|\leq 1$ and \begin{equation*}\chi(t)=1\text{ as }t<1\text{
      and }\chi(t)=0\text{ as }t>2,\end{equation*}

\item for $n=2$:
\begin{equation*}\chi\ie(x)=\
\begin{cases}
1&\text{as }\, |x-x\ie|\leq d\e,\\
\dfrac{\ln|x-x\ie|-2\ln \eps}{\ln d\e - 2\ln \eps}&
   \text{as }\, |x-x\ie|\in (d\e,\eps^2),\\
0&\text{as }\, |x-x\ie|\geq\eps^2,
\end{cases}\end{equation*}

\item $\widehat\chi\ie(x):= \chi\left(\dfrac{(2/\kappa)\cdot (\abs{x-x\ie}-d\e)}{\eps}\right)$,

\item for $n\geq 3$: $H\ie$ is the solution to the problem
  \begin{gather*}
    \begin{cases}
      \Delta H\ie(x)=0,&x\in\R^n\setminus\overline{D\ie},\\
      H\ie(x)=1,&x\in\partial D\ie,\\
      H\ie(x)\to 0,&|x|\to\infty
    \end{cases}
  \end{gather*}

\item for $n=2$:
  $H\ie$ is the solution to the problem
  \begin{gather*}
    \begin{cases}
      \Delta H\ie(x)=0,&x\in B_{1}(x\ie)\setminus\overline{D\ie},\\
      H\ie(x)=1,&x\in\partial D\ie,\\
      H\ie(x)=0,&x\in\partial B_{1}(x\ie),
    \end{cases}\\[1ex]\notag
    B_{1}(x\ie)\text{ is the unit ball centered at }x\ie,
  \end{gather*} 
  extended by $0$ to $\R^n\setminus \overline{B_{1}(x\ie)}$.
\end{itemize} 
Note, that the function $H\ie$ is defined on $\R^n\setminus {D\ie}$
(resp.\ $B_1(x\ie) \setminus D\ie$ if $n=2$).  We extend it onto
${D\ie}$ by $1$ (and onto $\R^n \setminus B_1(x\ie)$ by $0$ if $n=2$),
keeping the same notation $H\ie$.  Note also that $\|\nabla
H\ie\|_{\L(\R^n)}^2 = \capty(D\ie)$ by the definition of capacity
in~\eqref{capacity} and~\eqref{capacity2}.

We set $Y\ie:=\square\ie\setminus\overline{D\ie}$.
It is easy to see that 
\begin{gather*}
  P\ie f \restr {D\ie}
  = f \restr{D\ie} - f\ie,\quad 
  Q\ie f \restr {D\ie} = f\ie,
  \quad
\supp (P\ie f)\subset \square\ie,\quad \supp (Q\ie f)\subset \square \ie
\end{gather*}
(the inclusions are valid for $d\e\leq \kappa\eps$, which holds true
for small enough $\eps$ in view of \eqref{d}). Consequently $J\e^1
f\in \H_0(\Omega\e)$.

Now, we are in position to start the proof of~\eqref{cond1a}--\eqref{cond5}.

At first, we note that
conditions~\eqref{cond1b},~\eqref{cond2},~\eqref{cond3b} hold with
$\delta\e=0$, following from the definitions of the operators $J\e$,
$J'\e$ and $J^{1\prime}\e$.  Also, obviously, for each $f\in
\L(\Omega)$ and $u\in \L(\Omega\e)$ we have
\begin{equation*}
  \|J\e f\|_{\L(\Omega\e)}\leq \|f\|_{\L(\Omega)},\quad 
  \|J\e' u\|_{\L(\Omega)}= \|u\|_{\L(\Omega\e)},
\end{equation*} 
and therefore conditions~\eqref{cond4a}--\eqref{cond4b} are valid as well with $\delta_\eps=0$.
Thus, it remains to check the non-trivial conditions~\eqref{cond1a},~\eqref{cond3a} and~\eqref{cond5}.

The following Friedrichs- and Poincare-type inequalities will be frequently used  further.

\begin{lemma} 
One has
\begin{gather}\label{Friedrichs}
\forall v\in \H_0  (\square\ie):\quad 
\|v\|^2_{\L(\square\ie)}\leq C\eps^2 \|\nabla v\|^2_{\L(\square\ie)},\\\label{Poincare}
\forall v\in \H  (\square\ie):\quad 
\|v-\langle v\rangle_{\square\ie} \|^2_{\L(\square\ie)}\leq C\eps^2\|\nabla v\|^2_{\L(\square\ie)}.
\end{gather}

\end{lemma}

\begin{proof}
By the min-max principle 
\begin{gather*}
  \Lambda\e^\Dir
  =\min \left\{\frac{\|\nabla u\|^2_{\L(\square\ie)}}{\|u\|^2_{\L(\square\ie)}}:\ 
    u\in\H_0(\square\ie)\setminus\{0\}\right\},\\
  \Lambda\e^\Neu
  =\min\left\{\frac{\|\nabla u\|^2_{\L(\square\ie)}}{\|u\|^2_{\L(\square\ie)}}:\ 
    u\in\H(\square\ie)\setminus\{0\},\ \intl_{\square\ie}u(x)\dd x=0\right\},
\end{gather*}
where $\Lambda\e^\Dir$ (respectively, $\Lambda\e^\Neu$) is the first (respectively, the second)
eigenvalue of the Dirichlet (respectively, the Neumann) Laplacian on $\square\ie$.
Straightforward calculations gives 
\begin{gather*}
  \lambda\e^\Dir=n\left(\frac \pi \eps \right)^2,\qquad
  \lambda\e^\Neu= \left(\frac \pi \eps \right)^2,
\end{gather*} 
hence we easily get the required
inequalities~\eqref{Friedrichs}--\eqref{Poincare}.
\end{proof}

\subsection{Proof of~\eqref{cond1a}} Let  $f\in \L(\Omega)$. We have
\begin{gather}\label{1.1}
\|J\e f - J\e^1 f\|^2_{\HS\e}\leq
2\suml_{i\in\I}\left(\|P\ie f\|^2_{\L(\square\ie)} + \|Q\ie f\|^2_{\L(\square\ie)}\right).
\end{gather}

Using~\eqref{Poincare} and taking into account that $|\chi\ie(x)|\leq 1$, we obtain:
\begin{gather}\label{P-est}
\suml_{i\in\I}\|P\ie f\|^2_{\L(\square\ie)}\leq  \suml_{i\in\I}
\|f-f\ie\|^2_{\L(\square\ie)}\leq C\eps^2 \suml_{i\in\I}
\|\nabla f\|^2_{\L(\square\ie)}\leq C\eps^2\|f\|^2_{\HS^1}.
\end{gather}

Using~\eqref{Friedrichs} and taking into account that   
$|\widehat\chi\ie|\leq 1$, $|\widehat\chi\ie'|\leq C\kappa^{-1}\eps^{-1}$ and $\widehat\chi\ie\restr{D\ie}=H\ie\restr{D\ie}=1$ we obtain: 
\begin{multline}\label{1.2}
\suml_{i\in\I}\|Q\ie f\|^2_{\L(\square\ie)}\leq
 \suml_{i\in\I}|f\ie|^2 \|\widehat\chi\ie H\ie\|^2_{\L(\square\ie)}\leq
C\eps^2\suml_{i\in\I}
|f\ie|^2 \|\nabla(\widehat\chi\ie H\ie)\|^2_{\L(Y\ie)}\\
\leq
C_1\eps^2\suml_{i\in\I}|f\ie|^2\left(  \|\nabla H\ie\|^2_{\L(Y\ie)}+
 \kappa^{-2}\eps^{-2} \|H\ie\|_{\L(\supp(\nabla\widehat\chi\ie))}^2\right) .
\end{multline}
From~\eqref{q} and the definition of $H\ie$ we obtain the estimate  
\begin{gather}
  \label{1.3}
  \|\nabla H\ie\|^2_{\L(Y\ie)}
  \leq \|\nabla H\ie\|^2_{\L(\R^n)}
  = \capty(D\ie)
  \leq   C\eps^n.
\end{gather}
Using Lemma~\ref{lemma-Hest}, we obtain
\begin{equation*}
  |H\ie(x)|\leq C\tau_n(d\e)/ \tau_n(\eps),
  \quad\text{where}\quad
  \tau_n(r):=
  \begin{cases}
    r^{n-2},      &n\geq 3,\\
    |\ln r|^{-1}, &n=2,
  \end{cases}
\end{equation*} 
for $x\in \supp(\nabla\widehat\chi\ie)\subset\left\{x\in\R^n:\
  d\e+\kappa\eps/2 \leq |x-x\ie| \leq d\e+\kappa\eps\right\}$.  Hence,
taking into account that $\tau_n(d\e)\leq C\eps^n$ (see~\eqref{d}), we
deduce the asymptotics
\begin{gather}\label{1.4}
\|H\ie\|^2_{\L(\supp(\nabla\chi\ie))}=o(\eps^{n+2}).
\end{gather} 
Finally, applying the Cauchy-Schwarz inequality, one gets
\begin{gather}\label{1.5}
|f\ie|^2\leq \eps^{-n}\|f\|^2_{\L(\square\ie)}.
\end{gather} 
Combining~\eqref{1.2}--\eqref{1.5} we arrive at 
\begin{gather}\label{Q-est}
\suml_{i\in\I}\|Q\ie f\|^2_{\L(\square\ie)}\leq C_\kappa\eps^2\|f\|^2_{\HS}.
\end{gather}
Here and in what follows by $C_\kappa$ we denote a generic constant depending   on $\kappa$ and $n$.

From~\eqref{1.1},~\eqref{P-est},~\eqref{Q-est} we obtain
\begin{gather*}
\|J\e f - J\e^1 f\|_{\HS\e}\leq C_\kappa\eps\|f\|_{\HS^1}\leq \delta\e\|f\|_{\HS^1},
\end{gather*}
where $\delta\e$ is defined in~\eqref{delta}.  Therefore, we have
checked Condition~\eqref{cond1a}.
  
\subsection{Proof of~\eqref{cond3a}}

We need the following lemma, which was proven in~\cite[Lem.~4.9 and
Rem.~4.2)]{MK06}.

\begin{lemma} 
  \label{lemma2}
  Let $D\subset\R^n$ be a bounded convex domain, and let
  $D_1,D_2\subset D$ be measurable subsets with $|D_2|\not= 0$.  Then
  \begin{equation*}
    \|v\|^2_{\L(D_1)}
    \leq \frac {2|D_1|}{|D_2|}\|v\|^2_{\L(D_2)} +
    C \frac {(\diam(D))^{n+1}|D_1|^{1/n}}{|D_2|}\|\nabla v\|^2_{\L(D)}
  \end{equation*}
  for all $v\in \H(D)$, where $C$ depends only on the dimension $n$.
\end{lemma}

Let $f\in \L(\Omega)$. Applying Lemma~\ref{lemma2} with $v:=f$ and
$D:=\square\ie$, $D_1:=D\ie$, $D_2:=\square\ie$ we obtain
\begin{multline*}
\|f-J'\e J\e f\|^2_{\HS}=\suml_{i\in\I}\|f\|^2_{\L(D\ie)}
\leq 
C\suml_{i\in\I}\left(\left(\frac {d\e} \eps \right)^n\|f\|^2_{\L(\square\ie)}+ 
\frac{\eps^{n+1}d\e}{\eps^n}\|\nabla f\|^2_{\L(\square\ie)}\right)\\
\leq C\left(\left(\frac{d\e}\eps\right)^n+\eps d\e\right)\|f\|^2_{\HS^1}.
\end{multline*}
It is straightforward to show, using~\eqref{d} and the definition of
$\delta\e$ in~\eqref{delta}, that
\begin{equation*}
  \left(\frac{d\e}\eps\right)^n+\eps d\e
  \leq C\eps^2\leq C(\delta\e)^2
\end{equation*} 
and thus condition~\eqref{cond3a} is also valid.

\subsection{Proof of the form estimate~\eqref{cond5}} 

We will show that~\eqref{cond5} holds with $k=2$.  

Recall that $\HS^2$ is a Hilbert space consisting of functions from
$\dom(\A)=\mathsf{H}^2(\Omega)\cap \H_0(\Omega)$ with scalar product
$(f,g)_{\HS^2}:=(\A f+f ,\A g+g)_{\HS}$.

Since $\partial\Omega$ is $\mathsf{C}^2$-smooth, we can apply standard
elliptic regularity theory (see, e.g.,~\cite{GT01}): namely, the
$\HS^2$-norm is equivalent to the $\mathsf{H}^2(\Omega)$-norm, i.e.\
there is $C_\Omega>0$ such that for each $f\in\dom(\A)$ we have
\begin{gather}\label{h2h2}
\|f\|_{\mathsf{H}^2(\Omega)}\leq C_\Omega \|f\|_{\HS^2}
\end{gather}
(the fulfilment of \eqref{h2h2} for noncompact $\partial\Omega$ is due to the uniform regularity of $\Omega$). 

Note, that this is the only estimate in our proof in which the constant depends
on the domain $\Omega$. This results to $\Omega$-dependence of the constant standing in the definition of $\delta\e$.

Let $f\in \dom(\A)$, $u\in \H_0(\Omega\e)$. One has:
\begin{multline*} 
\left|\a\e[u,J\e^1 f]-\a[J\e^\pp u,f]\right|\\\leq
\left|\suml_{i\in\I}(\nabla u,\nabla P\ie f)_{\L(Y\ie)}\right|+ 
\left|\suml_{i\in\I}(\nabla u,\nabla Q\ie f)_{\L(Y\ie)}+q( u,f)_{\L(\Omega\e)}\right|=:
\mathcal{J}_{\eps,1}+\mathcal{J}_{\eps,2}.
\end{multline*}

\subsubsection{Estimates for $\mathcal{J}_{\eps,1}$}
 
One has, taking into account that $|\chi\ie(x)|\leq 1$:
\begin{multline} \label{J1.1}
\mathcal{J}_{\eps,1}\leq \Bigl|\suml_{i\in\I}(\nabla u, \chi\ie\nabla f)_{\L(\supp(\chi\ie))}\Bigr|+
\Bigl|\suml_{i\in\I}(\nabla u,(f-f\ie)\nabla \chi\ie)_{\L(\supp(\chi\ie))}\Bigr|\\
\leq \|\nabla u\|_{\L(\Omega\e)}\left(\Bigl(\suml_{i\in \I}\|\nabla f\|_{\L(\supp(\chi\ie))}^2\Bigr)^\frac12+ \Bigl(\suml_{i\in \I}\|(f-f\ie)\nabla \chi\ie\|_{\L(Y\ie)}^2\Bigr)^\frac12\right)
\end{multline}

Applying Lemma~\ref{lemma2} with $v:=\partial_j f$,
$D:=\square\ie$, $D_1:=\supp(\chi\ie)$ and $D_2:=\square\ie$, and taking
into account~\eqref{d} and~\eqref{h2h2} we obtain the estimates:
\begin{gather}\label{J1-est1}
\begin{array}{ll}
  n\geq 2:&
  \ds\suml_{i\in \I}\|\nabla f\|^2_{\L(\supp(\chi\ie))}
  \leq C\left(\left(\frac{d\e}\eps\right)^n
    +\eps \cdot d\e\right)\|f\|^2_{\mathsf{H}^2(\Omega)}
  \leq C\eps^2\|f\|^2_{\HS^2}, \\[1ex]
  n= 2:&\ds\suml_{i\in \I}\|\nabla f\|^2_{\L(\supp(\chi\ie))}
  \leq C\left(\left(\frac{\eps^2}\eps\right)^2
    +\eps\cdot \eps^2\right)\|f\|^2_{\mathsf{H}^2(\Omega)}
  \leq C\eps^2\|f\|^2_{\HS^2}.
\end{array}
\end{gather}
  
Now, we estimate the second term in~\eqref{J1.1}. 
One has the following H\"older-type inequality:
\begin{gather}\label{Hoelder}
  \forall p,q\in[2,\infty],\ 
  \frac 1p + \frac 1q
  =\frac12:\,\,\|(f-f\ie)\nabla \chi\ie\|_{\L(Y\ie)}\leq 
\|f-f\ie\|_{\mathsf{L}_p(Y\ie)} \|\nabla \chi\ie\|_{\mathsf{L}_q(Y\ie)}
\end{gather}
(for $p=\infty$ we use a convention $p^{-1}=0$).
Indeed, the classical H\"older inequality states that 
$\|FG\|_{\mathsf{L}_1(Y\ie)}\leq \|F\|_{\mathsf{L}_\mathbf{p}(Y\ie)}\|G\|_{\mathsf{L}_\mathbf{q}(Y\ie)}$ provided $\mathbf{p},\mathbf{q}\in[1,\infty]$, $\mathbf{p}^{-1}+\mathbf{q}^{-1}=1$.
Setting $\mathbf{p}:=p/2$, $\mathbf{q}:=q/2$, $F:=|f-f\ie|^2$, $G:=|\nabla \chi\ie|^2$ we easily arrive at~\eqref{Hoelder}.

One needs the following re-scaled Sobolev inequality.

\begin{lemma}\label{lemma-sobolev}
For each $f\in \mathsf{H}^2(\square\ie)$ 
\begin{gather}\label{sobolev}
\|f-f\ie\|_{\mathsf{L}_p(\square\ie)}\leq C_p\eps^{n/p+(2-n)/2}\|f\|_{\mathsf{H}^2(\square\ie)}
\end{gather}
provided $p$ satisfies 
\begin{gather}
  \label{p}
  2\leq p\leq \frac{2n}{n-4}\text{\; as\; }n\geq 5,\quad
2\leq p<\infty\text{\; as\; }n=4,\quad
p= \infty\text{\; as\; }n=2,3.
\end{gather}
The constant $C_p$ depends only on $p$.
\end{lemma}

\begin{proof}

Recall that $\square=(-1/2,1/2)^n$. Since $\mathsf{H}^2(\square)\hookrightarrow\mathsf{L}_p(\square)$ provided~\eqref{p} holds (see, e.g.,~\cite[Theorem~5.4]{A75})
one has for each $g\in \mathsf{H}^2(\square)$:
\begin{gather}\label{sob-square}
\|g\|_{\mathsf{L}_p(\square)}\leq C_p \|g\|_{\mathsf{H}^2(\square)}.
\end{gather}
Now, making the change of variables 
$\square\ni x= y\eps^{-1}-i\text{ with }y\in\square\ie$ in~\eqref{sob-square}, we infer from~\eqref{sob-square}:
for each $g\in\mathsf{H}^2(\square\ie)$
\begin{gather}\label{sob-square-ie}
  \eps^{-n/p}\|g\|_{\mathsf{L}_p(\square)}
  \leq C_p \Bigl(\eps^{-n}\|g\|^2_{\L(\square\ie)}
    +\eps^{2-n}\|\nabla g\|^2_{\L(\square\ie)}
    +\eps^{4-n}\suml_{k,l=1}^n\|\partial_{kl}^2g\|^2_{\L(\square\ie)}
  \Bigr)^{1/2}.
\end{gather}
Finally, we set $g:=f-f\ie$. Then, due to~\eqref{Poincare}, the estimate~\eqref{sob-square-ie} becomes
\begin{align*} 
 \|f-f\ie\|_{\mathsf{L}_p(\square)}
 &\leq C_p \eps^{n/p} \Bigl(\eps^{2-n}\|\nabla f\|^2_{\L(\square\ie)}
   +\eps^{4-n}\suml_{k,l=1}^n\|\partial_{kl}^2f\|^2_{\L(\square\ie)}
 \Bigr)^{1/2}\\
 &\leq C_p \eps^{n/p+(2-n)/2}\|f\|_{\mathsf{H}^2(\square\ie)}.
 \qedhere
\end{align*}
\end{proof}

We also need the estimate for $\chi\ie$, which is proved via a straightforward calculations. 
\begin{lemma}
  One has
  \begin{gather}
    \label{chi-est}
    \begin{array}{ll}
      \|\nabla\chi\ie\|_{\mathsf{L}_q(Y\ie)}\leq C (d\e)^{n/q-1},&n \ge 3,\ q\in [2,\infty],\\[1ex]
      \|\nabla\chi\ie\|_{\mathsf{L}_2(Y\ie)}\leq C |\ln d\e|^{-1/2},& n=2.
    \end{array} 
  \end{gather}
\end{lemma}
 
Now, we choose the largest $p$ for which~\eqref{p} holds:
\begin{gather}\label{pmax}
p:=\frac {2n}{n-4} \text{\, as }n\geq 5,\quad
p:=4\beta^{-1}\text{ with }\beta\in(0, 2)\text{\, as }n=4,\quad
p=\infty\text{\, as }n=2,3.
\end{gather}
As before
\begin{gather}\label{qmax}
q= \Bigl(\frac 12 -\frac 1p \Bigr)^{-1}.
\end{gather}

Plugging the estimates~\eqref{sobolev} and~\eqref{chi-est} into~\eqref{Hoelder} and taking into account~\eqref{d}, \eqref{pmax}--\eqref{qmax}  we arrive easily at 
\begin{gather}\label{J1-est2}
\begin{array}{ll}
  n\geq 5:&
  \ds\Bigl(\suml_{i\in\I}\|(f-f\ie)\nabla \chi\ie\|_{\L(Y\ie)}^2\Bigr)^{1/2}
  \leq C \cdot \frac{d\e} \eps\|f\|_{\mathsf{H}^2(\Omega)},
\\[2ex]
n= 4:&
\ds\Bigl(\suml_{i\in\I}\|(f-f\ie)\nabla \chi\ie\|_{\L(Y\ie)}^2\Bigr)^{1/2}\leq 
C_\beta\cdot \Bigl(\frac{d\e}\eps\Bigr)^{1-\beta}\|f\|_{\mathsf{H}^2(\Omega)}\leq C_\beta\eps^{1-\beta}\|f\|_{\mathsf{H}^2(\Omega)},
\\[2ex]
n= 3:&
\ds\Bigl(\suml_{i\in\I}\|(f-f\ie)\nabla \chi\ie\|_{\L(Y\ie)}^2\Bigr)^{1/2}\leq 
C\cdot \Bigl(\frac{d\e} \eps\Bigr)^{1/2}\|f\|_{\mathsf{H}^2(\Omega)}\leq C\eps \|f\|_{\mathsf{H}^2(\Omega)},
\\[2ex]
n= 2:&
\ds\Bigl(\suml_{i\in\I}\|(f-f\ie)\nabla \chi\ie\|_{\L(Y\ie)}^2\Bigr)^{1/2}\leq 
C|\ln d\e|^{-\frac12}\|f\|_{\mathsf{H}^2(\Omega)}\leq C\eps \|f\|_{\mathsf{H}^2(\Omega)}.
\end{array}
\end{gather}  


Combining~\eqref{J1-est1} and~\eqref{J1-est2} and taking into account
\eqref{h2h2} and the definition of $\delta\e$, we get the estimate
\begin{gather}\label{J1}
\mathcal{J}_{\eps,1}\leq \delta\e\|u\|_{\HS^1}\|f\|_{\HS^2}.
\end{gather}
 
\subsubsection{Estimates for $\mathcal{J}_{\eps,2}$}

One has
\begin{gather}\label{J2.0}
\mathcal{J}_{\eps,2}= \left|\suml_{i\in\I}f\ie(\nabla u, \nabla(H\ie \widehat \chi\ie))_{\L(Y\ie)}+q(u,f)_{\L(\Omega\e)}\right|.
\end{gather}

Besides~\eqref{capacity} (or~\eqref{capacity2} for $n=2$) there is another equivalent 
characterization of the capacity.

\begin{lemma}
  Let $D\subset \R^n$, and let $H$ be the solution of
  either~\eqref{cap-problem} if $n\geq 3$ or~\eqref{cap-problem2} if
  $n=2$.  Then
  \begin{gather}
    \label{cap-nu}
    \capty(D)
    =-\intl_{\partial D} \frac {\partial H}{\partial \nu}\dd s,
  \end{gather}
  where $\nu$ is the outward-pointing unit normal to $\partial D$,
  $\dd s$ is the area measure on $\partial D$.
\end{lemma}

\begin{proof}
For $n=2$ the result follows from
\begin{equation*}
  \intl_{\partial D} \frac {\partial H} {\partial \nu}\dd s
  =\intl_{\partial D}\frac {\partial H} {\partial \nu} H\dd s
  =-\intl_{B_1\setminus \overline{D}}|\nabla H|^2\dd x.
\end{equation*}
Here the first equality is due to $H\restr{\partial D}=1$, while
the second one is the Green formula, in which the surface integral over $\partial B_1$ vanishes since $H\restr{\partial B_1}=0$, and the second volume integral vanishes since $\Delta H=0$.

For $n\geq 3$ we proceed as follows. Let $B_R$ be the ball of radius $R>1$ being concentric with the smallest ball containing $D$. One has:
\begin{gather}
  \label{green}
  \intl_{\partial D} \frac {\partial H}{\partial \nu}\dd s
  =\intl_{\partial D}\frac {\partial H}{\partial \nu}H\dd s
  =-\intl_{B_R\setminus \overline{D}}|\nabla H|^2\dd x
  -\intl_{\partial B_R}\frac {\partial H}{\partial \nu}H\dd s
\end{gather}
(in the last integral $\nu$ is the inward-pointing unit normal to $\partial B_R$).
Lemma~\ref{lemma-Hest} implies the estimate 
\begin{gather}
  \label{green+}
  \left|\intl_{\partial B_R}\frac{\partial H}{\partial \nu}H\dd s\right|
  \leq \frac C{R^{n-2}}.
\end{gather}
Passing to the limit $R\to\infty$ in~\eqref{green} and taking into account~\eqref{capacity} and \eqref{green+} we arrive at the required equality~\eqref{cap-nu}.
The lemma is proved.
\end{proof}

Denote $u\ie:=\langle J\e' u\rangle_{\square\ie}$.
Integrating by parts and using~\eqref{cap-nu} we get
\begin{multline}\label{J2.1}
\suml_{i\in\I}f\ie(\nabla u, \nabla(H\ie \widehat \chi\ie))_{\L(Y\ie)}
\\=-\suml_{i\in\I}f\ie( u, \Delta(H\ie \widehat \chi\ie))_{\L(Y\ie)}=
-\suml_{i\in\I}f\ie u\ie( 1, \Delta(H\ie \widehat \chi\ie))_{\L(Y\ie)}+
\zeta\e\\=
\suml_{i\in\I}f\ie u\ie\intl_{\partial D\ie} \frac{\partial H\ie}{\partial\nu} \dd s 
+\zeta\e
= -\suml_{i\in\I}f\ie u\ie \capty(D\e)+\zeta\e,
\end{multline}
where $\zeta\e=\sum_{i\in\I}f\ie (u\ie-u, \Delta(H\ie \widehat
\chi\ie))_{\L(Y\ie)}$. Here we have used the facts that
$\widehat\chi\ie$ vanishes on $\left\{x\in Y\ie:\
  |x-x\ie|\geq d\e+\eps\kappa \right\}$ with all its derivatives,
$u\restr{\partial D\ie}=0$, and $\widehat\chi\ie(x)=1$ in a neighbourhood
of $D\ie$.

The remainder term $\zeta\e$ is small; namely the following estimate holds:
\begin{lemma}
One has:
\begin{gather}\label{zeta}
|\zeta\e|\leq 
C_\kappa\cdot\begin{cases}
\eps\|\nabla u\|_{\L(\Omega\e)}\|f\|_{\L(\Omega)},&n\geq 3,\\[1ex]
 \eps|\ln\eps|\|\nabla u\|_{\L(\Omega\e)}\|f\|_{\L(\Omega)},&n= 2.
\end{cases}
\end{gather}
\end{lemma}

\begin{proof}
At first we consider the case $n\geq 3$.
Since $\Delta H\ie=0$ we have  
\begin{gather*}\Delta(H\ie \widehat \chi\ie)=2\nabla H\ie\cdot\nabla \widehat \chi\ie+H\ie \Delta\widehat \chi\ie,\\ \supp(\Delta(H\ie \widehat \chi\ie))\subset\left\{x\in Y\ie:\ d\e+{\kappa\eps}/2  \leq |x-x\ie| \leq  d\e+ \kappa\eps   \right\},
\end{gather*}
hence, due to Lemma~\ref{lemma-Hest},~\eqref{d} and $|\nabla\widehat \chi\ie(x)|\leq C\kappa^{-1}\eps^{-1}$, $|\Delta\widehat \chi\ie(x)|\leq C\kappa^{-2}\eps^{-2}$, we get
\begin{gather}\label{Delta-est}
  |(\Delta(H\ie \widehat \chi\ie))(x)|
  \leq C\left(\frac {(d\e)^{n-2}}{\eps^{n-1}}\cdot \frac 1{\kappa\eps}
    +\frac{(d\e)^{n-2}}{\eps^{n-2}} \cdot \frac 1{\kappa^2\eps^2}\right)\leq C_\kappa,
\end{gather}
Then, using  the Cauchy-Schwarz inequality,~\eqref{Poincare} and~\eqref{Delta-est}, we obtain
\begin{gather*}
|\zeta\e|\leq C_\kappa\left(\suml_{i\in\I}\|J\e'u-u\ie\|_{\L(\square\ie)}\right)^{1/2}\left(\suml_{i\in\I}|f\ie|^2\eps^n\right)^{1/2}\leq C_\kappa\eps\|\nabla u\|_{\L(\Omega\e)}\|f\|_{\L(\Omega)}.
\end{gather*}

In the case $n=2$ Lemma~\ref{lemma-Hest} gives 
\begin{gather*}
|(\Delta(H\ie \widehat \chi\ie))(x)|\leq C_\kappa|\ln\eps|,
\end{gather*}
and, via the same arguments as in the case $n\geq 3$, we obtain
\begin{gather*}
|\zeta\e|\leq C_\kappa\eps|\ln\eps|\|\nabla u\|_{\L(\Omega\e)}\|f\|_{\L(\Omega)}.
\end{gather*}
The lemma is proved.
\end{proof}

We need also the estimate for $\H_0(\Omega)$-functions in a neighbourhood of $\partial\Omega$.
\begin{lemma}\label{lemma-te} 
Let $T\e=\Omega\e\setminus\overline{\cup_{i\in\I}Y\ie}$.
One has for each $f\in \H_0(\Omega)$:
\begin{gather}\label{te}
\|f\|_{\L(T\e)}\leq C\eps\|\nabla f\|_{\L(\Omega)}.
\end{gather}
\end{lemma} 

\begin{proof}
We denote   
\begin{equation*}
\widehat{\Omega}\e=\{y\in\R^n:\ y= x + t \nu(x),\ x\in\partial\Omega,\ t\in[0,\sqrt{n}\eps)\}
\end{equation*}
(recall that $\nu \colon \partial\Omega\to \mathbb{S}^{n-1}$ is a unit inward-pointing normal vector field on $\partial\Omega$).
Note, that $\sqrt{n}$ is the length of the diagonal of the cube $\square$. Taking this into account one can easily deduce from the definition of the set $\I$ that $T\e\subset \widehat{\Omega}\e$.
  
Let $\Delta_{\widehat{\Omega}\e}$ be the Laplace operator on $\widehat{\Omega}\e$ subject to the Dirichlet conditions on $\partial\Omega$ and the Neumann conditions on $\partial \widehat{\Omega}\e\setminus\partial\Omega$. 
One has the following asymptotic equality (see ~\cite{Kr14}):
\begin{equation*}
  \inf \sigma(-\Delta_{\widehat{\Omega}\e})
  =\left(\frac \pi {2 \sqrt{n}\eps}\right)^2+\frac{C_\Omega} \eps +\mathcal{O}(1),\end{equation*}
where the constant $C_\Omega$ depends on the principal curvatures of $\partial\Omega$.
Note, that the result of~\cite{Kr14} is obtained under the assumption that the map~\eqref{tau}  is injective on $\partial\Omega \times [0, \sqrt{n}\eps]$, that indeed holds true provided $\eps$ is small enough, namely  
$\eps< \theta_\Omega/\sqrt{n}$.

Hence, using the minimax principle, we get the inequality
\begin{gather}\label{te+}
\|f\|^2_{\L(\widehat{\Omega}\e)}\leq C\eps^2\|\nabla f\|^2_{\L(\widehat{\Omega}\e)},
\end{gather}
which holds for each $f\in \H(\widehat{\Omega}\e)$ with $f\restr{\partial\Omega}=0$. 
Obviously,~\eqref{te} follows from~\eqref{te+}.
\end{proof}
 
Using~\eqref{Poincare},~\eqref{J2.1},~\eqref{te} we obtain from
\eqref{J2.0}:
\begin{multline*}
\mathcal{J}_{\eps,2}\leq |\zeta\e|+
\left|\suml_{i\in\I}f\ie u\ie\capty(D\e)-q\suml_{i\in\I}(f,u)_{\L(Y\ie)}\right|+q\left|(f,u)_{\L(T\e)}\right|\\\leq
|\zeta\e|+
\left|(\capty(D\e)\eps^{-n}-q)\suml_{i\in\I}\eps^n f\ie u\ie\right|+q\left|\suml_{i\in\I}(f\ie-f,u)_{\L(Y\ie)} \right|+q\left|(f,u)_{\L(T\e)}\right|\\\leq
|\zeta\e| +|\capty(D\e)\eps^{-n}-q|\|f\|_{\L(\Omega\e)}\|u\|_{\L(\Omega\e)}+
q\left(\suml_{i\in\I}\|f-f\ie\|_{\L(\square\ie)}\right)^{1/2}\|u\|_{\L(\Omega\e)}\\+q\|f\|_{\L(T\e)}\|u\|_{\L(T\e)}\leq 
|\zeta\e| +|\capty(D\e)\eps^{-n}-q|\|f\|_{\L(\Omega\e)}\|u\|_{\L(\Omega\e)}\\+
C\eps\|\nabla f\|_{\L(\Omega)}\|u\|_{\L(\Omega\e)}+q\|f\|_{\L(T\e)}\|u\|_{\L(T\e)},
\end{multline*}
hence, taking into account~\eqref{zeta}, we  get
\begin{gather}\label{J2}
\mathcal{J}_{\eps,2}\leq \delta\e\|u\|_{\HS^1}\|f\|_{\HS^1}.
\end{gather}

Combining estimates~\eqref{J1} and~\eqref{J2} we obtain~\eqref{cond5} with $k=2$.

Thus, we have checked the fulfilment of conditions~\eqref{cond1a}--\eqref{cond5}, hence we immediately get  Theorems~\ref{thm:main:1}--\ref{thm:main:3}.

\begin{remark}\label{remark-smooth}

It is evident from the proof that the assumptions on $\partial\Omega$ can be 
weakened. We use them twice: to guarantee the fulfilment of~\eqref{h2h2} (elliptic regularity) and to prove estimate~\eqref{te+}, where we utilize the result from~\cite{Kr14}. 
It is well-known, that the elliptic regularity is still valid under less restrictive assumptions, for example, if $\partial\Omega$ is compact and belongs to $\mathsf{C}^{1,1}$ class or $\Omega$ is a convex domain with Lipschitz boundary  (see, e.g.,~\cite[Theorems~2.2.2.3 and 3.2.1.2]{Gr85}). Apparently,  inequality~\eqref{te+} can be proved for Lipschitz domains under additional restrictions on  principal curvatures.
\end{remark}

\subsection{Proof of Theorem~\ref{thm:main:4}}

We will use the results of~\cite{IOS89}. 
Let $\HS\e$ and
$\HS$ be separable Hilbert spaces, and
$B\e \colon \HS\e\to \HS\e,\
B \colon \HS \to \HS$ be linear
compact self-adjoint positive operators.
We denote by
 $\{\mu_{k,\eps}\}_{k\in\N}$ and
$\left\{\mu_k\right\}_{k\in\N}$ the eigenvalues of the
operators $B\e$ and $B$, respectively,
being
renumbered in the descending order and with account of their
multiplicity. 

\begin{theorem}[\cite{IOS89}]
  \label{thm:IOS}
  Assume that the following conditions $A_1-A_4$ hold:

{$A_1.$} The linear bounded operator $J\e \colon \HS\to
\HS\e$ exists such that for each $f\in \HS$
\begin{gather*}
\|J\e
f\|_{\HS\e} \to
\|f\|_{\HS}\text{ as }\eps\to 0.
\end{gather*}

{$A_2.$} The norms
$\|B\e\|_{\mathcal{L}(\HS\e)}$ are bounded
uniformly in $\eps$.

{$A_3.$} For any $f\in\HS$: $\|B\e J\e
f-J\e B f\|_{\HS\e}\to 0 \text{ as }\eps\to 0$.

{$A_4.$} For any family $\{f\e\in \HS\e\}_\eps$ with $\sup_{\eps}
\|f\e\|_{\HS\e}<\infty$ there exist a sequence $(\eps_m)_m$ and $w\in
\HS$ such that $ \|B_{\eps_m} f_{\eps_m}-J_{\eps_m}
w\|_{\HS_{\eps_m}}\to 0$ and $ \eps_m\to 0$ as $m \to \infty$.

Then for any $k\in\mathbb{N}$ we have
\begin{equation*}
  |\mu_{k,\eps}-\mu_k|\leq 
  C\e \sup\limits_f\|B\e J\e f-J\e B f\|_{\HS\e},
\end{equation*}
where $|C\e|\leq C$, $\lim_{\eps\to 0}C\e=1$, the supremum is taken over all
$f\in\HS$ belonging to the eigenspace associated with $\mu_k$ and
satisfying $\|f\|_\HS=1$.
\end{theorem}

We apply this theorem with $B\e=(\A\e+I)^{-1}$, $B=(\A+I)^{-1}$.
These operators are positive, self-adjoint and compact (recall that
$\Omega$ is a bounded domain here), moreover
$\|B\e\|_{\mathcal{L}(\HS\e)}\leq 1$. Thus condition~$A_2$ is
fulfilled. We choose the operator $J\e$ by~\eqref{J}; due
to~\eqref{totalD} condition~$A_1$ is valid. By
Theorem~\ref{thm:main:1} condition~$A_3$ holds as well. Finally, since
$\|B_\eps\|_{\mathcal{L}(\HS\e)} \le 1$, the set $\{\|B\e
f\e\|_{\H(\Omega\e)}\}_\eps$ is also bounded.  Then the sequence
$\{J'\e B\e f\e\}_\eps$ is bounded in $\H(\Omega)$ (recall that the
operator $J\e'$ is defined in~\eqref{J'}), and by Rellich's embedding
theorem it is compact in $\L(\Omega)$ provided $\Omega$ is bounded.
Thus there exist $w\in \L(\Omega)$ and a sequence $(\eps_m)_m$ such
that $\|J'_{\eps_m} B_{\eps_m} f_{\eps_m}-w\|_{\L(\Omega)}\to 0$ and
$\eps_m\to 0$ as $m \to \infty$, hence we immediately obtain
Condition~$A_4$.

Combining Theorems~\ref{thm:main:1} and~\ref{thm:IOS} we arrive at the
estimate
\begin{equation}\label{lambda-est1}
\left|\mu_{k,\eps}-\mu_k\right|\leq 4C\e\delta\e,
\end{equation} 
where $|C\e|\leq C$, $\lim_{\eps\to 0}C\e=1$ and $\delta\e$ is given in \eqref{delta}. 
Since $\mu_{k,\eps}=(\lambda_{k,\eps}+1)^{-1}$, and $\mu_k=(\lambda_{k}+1)^{-1}$, \eqref{lambda-est1} is equivalent to \eqref{spectrum2}.

Finally, we observe that for each fixed $k\in\mathbb{N}$
\begin{gather}\label{lambda-est2}
\lambda_{k,\eps}\leq C_k
\end{gather}
that follows  from Theorem~\ref{thm:abstract:4} and
Remark~\ref{remark-hausd} (otherwise, we will easily obtain a
contradiction with Condition~(ii) from this
remark). \eqref{spectrum2}, \eqref{lambda-est2} imply \eqref{spectrum1}. Theorem~\ref{thm:main:4} is proved.



\providecommand{\bysame}{\leavevmode\hbox to3em{\hrulefill}\thinspace}
\providecommand{\MR}{\relax\ifhmode\unskip\space\fi MR }
\providecommand{\MRhref}[2]{%
  \href{http://www.ams.org/mathscinet-getitem?mr=#1}{#2}
}
\providecommand{\href}[2]{#2}

\end{document}